\documentclass[reqno]{amsart}

\usepackage{amsmath}
\usepackage{amsfonts}
\usepackage{amssymb,enumerate}
\usepackage{amsthm}
\usepackage{cite}
\usepackage{comment}
\usepackage[all]{xy}
\usepackage{hyperref}
\usepackage{color}

\theoremstyle{plain}
\newtheorem{lem}{Lemma}[section]
\newtheorem{cor}[lem]{Corollary}
\newtheorem{prop}[lem]{Proposition}
\newtheorem{thm}[lem]{Theorem}

\theoremstyle{definition}
\newtheorem{defn}[lem]{Definition}
\newtheorem{ex}[lem]{Example}

\newtheorem{disc}[lem]{Remark}

\newtheorem*{convention*}{Convention}

\newtheorem*{assumption}{Assumption}
\newtheorem*{assumptions}{Assumptions}

\newcommand{\cat}[1]{\mathcal{#1}}

\newcommand{\catw}{\cat{W}}

\newcommand{\catf}{\cat{F}}

\newcommand{\catc}{\cat{C}}
\newcommand{\catd}{\cat{D}}

\newcommand{\id}{\operatorname{id}}

\newcommand{\im}{\operatorname{Im}}

\newcommand{\ti}{\tilde}

\newcommand{\bbz}{\mathbb{Z}}
\newcommand{\bbn}{\mathbb{N}}

\newcommand{\from}{\leftarrow}
\newcommand{\xra}{\xrightarrow}
\newcommand{\xla}{\xleftarrow}

\newcommand{\emptytuple}{\mathfrak{0}}

\renewcommand{\geq}{\geqslant}
\renewcommand{\leq}{\leqslant}

\newcommand{\mor}[3]{\operatorname{Mor}_{#1}(#2,#3)}

\newcommand{\catset}{\mathbf{Set}}
\newcommand{\catcat}{\mathbf{Cat}}
\newcommand{\catdom}{\mathbf{Dom}}
\newcommand{\catpom}{\mathbf{POM}}

\numberwithin{equation}{lem}

\newcommand{\divs}{\mid_{\text w}^{}}

\newcommand{\zaks}[1]{\ell_{\text{#1}}}
\newcommand{\zaksmor}{\zaks{mor}}
\newcommand{\zaksobj}{\zaks{obj}}
\newcommand{\zakselt}{\zaks{elt}}
\newcommand{\objd}{\operatorname{Obj}(\catf(D \setminus \{0\}))}
\newcommand{\mord}{\operatorname{Mor}(\catf(D \setminus \{0\}))}

\makeatletter
\newcommand*{\doublerightarrow}[2]{\mathrel{
  \settowidth{\@tempdima}{$\scriptstyle#1$}
  \settowidth{\@tempdimb}{$\scriptstyle#2$}
  \ifdim\@tempdimb>\@tempdima \@tempdima=\@tempdimb\fi
  \mathop{\vcenter{
    \offinterlineskip\ialign{\hbox to\dimexpr\@tempdima+1em{##}\cr
    \rightarrowfill\cr\noalign{\kern.5ex}
    \rightarrowfill\cr}}}\limits^{\!#1}_{\!#2}}}
\newcommand*{\triplerightarrow}[1]{\mathrel{
  \settowidth{\@tempdima}{$\scriptstyle#1$}
  \mathop{\vcenter{
    \offinterlineskip\ialign{\hbox to\dimexpr\@tempdima+1em{##}\cr
    \rightarrowfill\cr\noalign{\kern.5ex}
    \rightarrowfill\cr\noalign{\kern.5ex}
    \rightarrowfill\cr}}}\limits^{\!#1}}}
\makeatother

\usepackage[font=footnotesize, labelfont={sf,bf}, margin=1cm]{caption}

\begin{document}

\author{Brandon Goodell}
\address{Department of Mathematical Sciences\\
	Clemson University\\
O-110 Martin Hall, Box 340975, Clemson, S.C. 29634, USA}
\email[Goodell]{bggoode@g.clemson.edu}

\author{Sean K. Sather-Wagstaff}
\email[Sather-Wagstaff]{ssather@clemson.edu}
\urladdr[Sather-Wagstaff]{https://ssather.people.clemson.edu/}




\date{\today}

\title{The Category of Factorization}
\begin{abstract}
We introduce and investigate the category of factorization of a multiplicative, commutative, cancellative, pre-ordered monoid $A$, which we denote $\mathcal{F}(A)$. The objects of $\mathcal{F}(A)$ are factorizations of elements of $A$, and the morphisms in $\mathcal{F}(A)$ encode combinatorial similarities and differences between the factorizations. We pay particular attention to the divisibility pre-order and to the monoid $A=D\setminus\{0\}$ where $D$ is an integral domain.

Among other results, we show that $\mathcal{F}(A)$ is a symmetric and strict monoidal category with weak equivalences and compute the associated category of fractions obtained by inverting the weak equivalences. Also, we use this construction to characterize various factorization properties of integral domains: atomicity, unique factorization, and so on.
\end{abstract}

\maketitle

\tableofcontents

\section{Introduction}\label{sec-intro}

It is well known that factorization in an integral domain can be poorly behaved, e.g., via the 
failure of unique factorization.
Various constructions and techniques have been used to understand the variety of behaviors that can occur. 
An example of this is the group of divisibility, which goes back at least to Krull~\cite{krull:ab},
and serves as important motivation for the present work.
See also, e.g., \cite{G3,jaffard,lorenzen,mockor,ohm,rump}.

In this paper, we introduce a construction that tracks all factorizations 
of non-zero elements
in an integral domain $D$ 
simultaneously:
the category of factorization. 
The objects of this category are the finite sequences/ordered tuples $(x_n)_{n=1}^N=(x_n)$ of non-zero elements of $D$,
along with the empty tuple $\emptytuple$. 
Each sequence $(x_n)$ can be thought of as a factorization of the product $x_1\cdots x_N$,
with the product of the empty sequence (i.e., the empty product) defined to be 1. 
In 
other words, the objects in this category are exactly the finite 
factorizations in
$D\setminus\{0\}$.

We specify the morphisms in this category in Definition~\ref{defn180217a}.
Big picture, they track various divisibility properties between factorizations. 
As an example, one type of morphism comes from  factoring the entries of a tuple. 
For instance, over $D=\bbz$ we have the tuple $(6,35)$. The factorizations $6=2\cdot 3$ and $35=5\cdot 7$
yield a morphism $(6,35)\to(2,3,5,7)$ as well as morphisms $(6,35)\to(2,5,3,7)$ and $(6,35)\to(1,7,-1,2,3,5)$ and so on. 
We identify other important types of morphisms that monitor other information about the factorizations
in Example~\ref{basic-example} and Definition~\ref{defn170715b}.

Note that we do not identify factorizations that one would usually think of as essentially the same, e.g., 
$(2,3)\neq(3,2)\neq(-3,-2)$. However, these tuples are 
isomorphic
in this category.
In fact, two tuples $(x_n)$ and $(y_m)$ are isomorphic if and only if they have the same length and there is a permutation
$\sigma$ such that 
for all $i$, 
the element $y_i$ is a unit multiple of
$x_{\sigma(i)}$.
Thus, the morphisms in this category are defined so
that the isomorphisms exactly monitor  which factorizations  are essentially the same.

As in many papers in the subject of factorization in integral domains,
we work in the more general
setting of a cancellative commutative monoid, periodically specializing to the particular case of 
the non-zero elements of an integral domain. 
See Section~\ref{sec-def} for the formal definition 
of our category
as well as some examples.
Section~\ref{sec170415b} contains fundamental properties of the morphisms in this category,
including characterizations of the monic morphisms, epic morphisms, and isomorphisms.
Section~\ref{sec170415a} constructs several useful functors; in particular, we construct adjoint functors
between the category of factorization and the original monoid, 
and we show in Theorem~\ref{prop190106b} that the category of factorization is a symmetric and strict monoidal category.

In 
Section~\ref{model-category} we 
introduce a class of weak equivalences
that mirrors properties of 
invertible
elements of $A$. 
Thus, the morphisms in $\catf(A)$ not only track relations between
factorizations in $A$, but they also see factorization properties of individual elements of $A$,
a theme that we revisit in subsequent sections.
In Theorem~\ref{thm170501a}, we show that the 
category 
of fractions obtained by inverting the weak equivalences
is exactly the original monoid $A$, considered as a category in a natural way.

Section~\ref{sec170708a} is a treatment of a notion of 
weak divisibility
for morphisms in $\catf(A)$.
Our main purpose here
is to be able to define and investigate
what it means for a morphism to be 
weakly prime
(the subject of
Section~\ref{sec170702a}) as a companion to the 
weakly irreducible
morphisms of
Section~\ref{sec170501a}. As the morphisms in our category encode relations between factorizations,
weak divisibility of morphisms encodes divisibility conditions between these relations
by Theorem~\ref{prop170708a}. And
the 
weakly prime and weakly irreducible
morphisms encode prime and irreducible relations (loosely speaking), 
by Theorems~\ref{lem170501a} and~\ref{lem170501azzz},
respectively.

The paper ends with
Section~\ref{accp-implies-atomic} which 
includes several characterizations of factorization properties of integral domains: 
unique factorization, half-factorization and so on. 
For instance, part of Theorem~\ref{prop170503a} states that an integral domain is atomic
if and only if every morphism $(x_n)\to(y_m)$ in $\catf(D \setminus \{0\})$ between non-empty tuples decomposes as 
a finite composition
$(x_n)\to\cdots\to(y_m)$ of 
weakly irreducible
morphisms and weak equivalences; this is parallel to the definition in terms of 
decompositions of elements 
as products
of 
irreducible elements
and units.

One application of the ideas from this paper is  in a new class of 
integral domains, the ``irreducible divisor pair domains'' or IDPDs.
See~\cite{sather:idpd} for more on these.

\begin{assumptions}
Throughout this paper, let $(A,\leq)$ be a multiplicative, commutative, cancellative, pre-ordered monoid. 
In particular, the relation $\leq$  is reflexive and transitive
(though it is neither symmetric nor antisymmetric in general), and it respects the monoid operation.
Sometimes, we specify that $A$ is a \emph{divisibility monoid}; this means that 
$A$ is a multiplicative, commutative, cancellative monoid where
the partial order is the \emph{divisibility order} 
$a\mid b$.
Let  $\catpom$ denote the category of pre-ordered monoids where morphisms are 
monoid homomorphisms $\phi\colon A\to B$ that respect the pre-orders on $A$ and $B$.

We set $\bbn = \left\{1,2,\ldots\right\}$. 
For all $N\in\bbn$, set $[N]=\{1,2,\ldots,N\}$.
We denote the identity function on a set $X$ as $\text{id}_X$.  
Also, $\catcat$ is the category of small categories and $\catset$ is the category of sets.
\end{assumptions}

\section{Definitions and Examples}\label{sec-def}

We begin this section by defining
the category of factorization.

\begin{defn}\label{defn180217a}
The \textit{category of factorization} of the monoid $A$, denoted $\mathcal{F}(A)$ has 
as its objects
all ordered tuples of monoid elements
\[\text{Obj}(\mathcal{F}(A)) = \left\{(x_n)_{n=1}^{N} \mid N \in \bbn, x_n \in A\right\} \cup \left\{\emptytuple\right\}\] 
including the \emph{empty tuple} denoted $\emptytuple$.
(In other words, $\text{Obj}(\mathcal{F}(A))$ is the set underlying the free monoid on $A$. See Definition~\ref{defn170605a} for more
about this.)
Again, we follow the convention that the empty product is $1=1_A$.
When we see no danger of confusion, we write a tuple $(x_n)_{n=1}^{N}$ as $(x_n)$, we write $(y_m)_{m=1}^{M}$ as $(y_m)$,
and so on. 

The morphisms $(x_n)\to(y_m)$ correspond to the functions $f\colon[M]\to[N]$ that
are \emph{order-constrained} by  $(x_n)$ and $(y_m)$, by which we mean that
for every $n\in[N]$ we have $x_n \leq \prod_{m \in f^{-1}(n)} y_m$. 
Given such a function $f$, we usually denote the associated morphism in $\mathcal F(A)$ as
$\hat f\colon (x_n)\to(y_m)$.
Since the function $f\colon[M]\to[N]$ can represent many different morphisms
in $\catf(A)$ with different domains or codomains, we occasionally employ the notation
$\hat f^{(x_n)}_{(y_m)}\colon (x_n)\to(y_m)$, but only when necessary.

Two morphisms $\hat f,\hat h\colon (x_n)\to(y_m)$ are equal if $f=h$. We
define the composition of two morphisms $\hat{f} \colon(x_n)\to (y_m)$ 
and $\hat{g} 
\colon (y_m)\to (z_p)$ by setting $\hat{g}\circ \hat{f} = \widehat{f \circ g}$. 
\label{cat-defined}
\end{defn}

We could simply define
$\mor{\catf(A)}{(x_n)}{(y_m)}$
to be
the set of 
order-constrained
functions $f\colon [M]\to[N]$,
i.e.,
we could define
$\mor{\catf(A)}{(x_n)}{(y_m)}$
as a subset of $\mor{\catset}{[M]}{[N]}$.
However, we find it helpful to distinguish between $f$ and $\hat f$.

In Theorem~\ref{cat-is-small-cat} below, 
we prove that $\mathcal{F}(A)$ is a 
small
category under this definition of composition of morphisms.
Before that, though, we discuss examples like those in the introduction to help familiarize the reader with our construction.
In particular, we introduce some terminology informally in this example, which we formalize in Definition~\ref{defn170715b} below.

\begin{ex}\label{basic-example}
Consider the divisibility monoid $A = \bbz\setminus \left\{0\right\}$. 
The objects of $\mathcal{F}(A)$ are ordered tuples of non-zero integers. 

Some morphisms in $\mathcal{F}(A)$ are ``factorization morphisms'' corresponding to factoring the elements $x_n$.
For example, we have a morphism 
$\hat f\colon (6,35) \to (2,3,5,7)$ 
corresponding to the equalities $6 = 2\cdot 3$ and $35 = 5 \cdot 7$. 
Formally, with $(x_n)=(6,35)$ and $(y_m)=(2,3,5,7)$ we have $N=2$ and $M=4$, and $\hat f$ corresponds to the function
$f\colon [4]\to[2]$ given by $f(1) = f(2) = 1$ and $f(3) = f(4) = 2$.
Moreover, $\hat f$ is the unique morphism $(6,35) \to (2,3,5,7)$; indeed, there are $2^4$  functions $[4] \to [2]$, 
but one checks readily that 
the given
function $f$ is the only one that is
order-constrained, so it is the only one that determines a morphism in $\catf(A)$. 

Some morphisms are ``divisibility morphisms'',
e.g., the morphism  $\hat g\colon(2,5) \to (6,15)$ corresponding to the conditions 
$2 \mid 6$ and $5 \mid 15$. 
Formally, this comes from the identity function $g=\id_{[2]}\colon [2] \to [2]$. 
Such morphisms  
are similar to  factorization morphisms in that no tuple elements are ``dropped'' from the domain objects since $f$ and $g$ are surjective.

Some morphisms ``drop invertible coordinates''. For instance, we have a morphism $\hat h\colon(6,1,1) \to (6)$,
coming from the function $h\colon [1] \to [3]$ with $h(1) = 1$. 
The point here is that $h^{-1}(2)=\emptyset$, so for $h$ to be order-constrained
requires the second entry $1$ to divide the empty product, and so on.
On the other hand, 
there are no morphisms $(6,2,1)\to (6)$ nor $(6,2,1)\to(6,1)$. Indeed, any such morphism would require 6 or 2 to divide the empty product 1.
These are similar to divisibility morphisms in that the underlying functions $g$ and $h$ are injective.

Most morphisms in $\catf(A)$ exhibit aspects of all of the above  types. 
For example, consider the morphism 
$\widehat\alpha\colon(2,3,1,1) \to (2,7,3,5)$ coming
from the function $\alpha\colon[4] \to [4]$ defined by $\alpha(1)=\alpha(2)=1$ and $\alpha(3)=\alpha(4)=2$.
One can decompose $\widehat\alpha$ as the composition
\[(2,3,1,1) \xra{\hat\epsilon} (2,3) \xra{\widehat{\id_{[2]}}} (14,15) \xra{\hat\phi} (2,7,3,5)\]
where $\hat\epsilon$ drops the invertible coordinates, 
$\widehat{\id_{[2]}}$ is a divisibility morphism that uses the  relations $2 \mid 14$ and $3 \mid 15$, 
and $\hat\phi$ is a factorization morphism that uses
the factorizations $14=2\cdot 7$ and $15=3\cdot 5$. 
See Proposition~\ref{prop170423a} for a general factorization result for morphisms.
\end{ex}

\begin{thm}\label{cat-is-small-cat}
The construction
$\mathcal{F}(A)$ is a small category.
\end{thm}

\begin{proof}
We need to show that composition is well-defined in $\catf(A)$.
That is, given morphisms $\hat f\colon (x_n)\to(y_m)$ and $\hat g\colon(y_m)\to(z_p)$, we need to show that
$\widehat{f\circ g}$ is a morphism in $\catf(A)$. For each $n\in[N]$ we have
$x_n \leq \prod_{m \in f^{-1}(n)} y_m$ since $f$ is order-constrained,
and similarly for each $m\in f^{-1}(n)$ we have
$y_m \leq \prod_{p \in g^{-1}(m)} z_p$.
Since the partial order on $A$ respects the operation on $A$, it follows that
$$\textstyle
x_n \leq \prod_{m \in f^{-1}(n)} y_m\leq \prod_{m \in f^{-1}(n)}\left(\prod_{p \in g^{-1}(m)} z_p\right)=\prod_{p \in (f\circ g)^{-1}(n)} z_p$$
so $\widehat{f\circ g}$ is indeed a morphism.

It is straightforward to show that each identity function $\id_{[N]}\colon[N]\to[N]$ 
is order-constrained with respect to $(x_n)$, determining a morphism $(x_n)\to(x_n)$. 
This yields an identity morphism since, for instance,
$\hat f\circ\widehat{\id_{[N]}}=\widehat{\id_{[N]}\circ f}=\hat f$.
Also, since composition in $\catf(A)$ is based on composition of functions in $\catset$, our rule of composition is
associative, so $\catf(A)$ is a category.

By construction, the collection of objects of $\catf(A)$ is a set, being the collection of all ordered tuples over the set $A$.
Also, each collection $\mor{\catf(A)}{(x_n)}{(y_m)}$ is a finite set since it corresponds to a subcollection of the set
$\mor{\catset}{[M]}{[N]}$.
Thus, $\catf(A)$ is a small category, as desired.
\end{proof}

For the last result of this section, recall the notations $\catcat$ and $\catpom$ from our assumptions at the end of 
Section~\ref{sec-intro}.

\begin{prop}\label{prop170430d}
The category of factorization operation is a covariant functor $\catf\colon\catpom\to\catcat$.
\end{prop}

\begin{proof}
Let $\phi\colon A\to B$ be a morphism in $\catpom$. 
Let $\catf(\phi)\colon\catf(A)\to\catf(B)$ be given on objects as
$\catf(\phi)(x_n)=(\phi(x_n))$. 
For each morphism $\hat f^{(x_n)}_{(y_m)}\colon (x_n)\to (y_m)$ in $\catf(A)$,
take the  relation $x_n\leq\prod_{m\in f^{-1}(n)}y_m$ and apply $\phi$ to conclude that
$\phi(x_n)\leq\phi\left(\prod_{m\in f^{-1}(n)}y_m\right)=\prod_{m\in f^{-1}(n)}\phi(y_m)$;
from this, we conclude that $\hat f^{(\phi(x_n))}_{(\phi(y_m))}\colon(\phi(x_n))\to(\phi(y_m))$ is a morphism in $\catf(B)$, in other words,
we have that
$\hat f^{\catf(\phi)(x_n)}_{\catf(\phi)(y_m)}\colon\catf(\phi)(x_n)\to\catf(\phi)(y_m)$ is a morphism in $\catf(B)$,
and we define $\catf(\phi)(\hat f^{(x_n)}_{(y_m)})=\hat f^{\catf(\phi)(x_n)}_{\catf(\phi)(y_m)}$.
It is straightforward to show that $\catf(\id_A)=\id_{\catf(A)}$ and
$\catf(\phi\circ\psi)=\catf(\phi)\circ\catf(\psi)$, so $\catf$ is a 
covariant functor, as desired.
\end{proof}

\begin{disc}\label{disc170430d}
Many examples of interest for us involve taking an integral domain $D$ and considering the divisibility monoid $A=D\setminus\{0\}$
and its category of factorization $\catf(D\setminus\{0\})$. 
If $\catdom$ is the category of integral domains where the morphisms are the ring monomorphisms, then the operation
$D\mapsto D\setminus\{0\}$ describes a covariant functor $\catdom\to\catpom$,
where a morphism $\iota\colon D\hookrightarrow D'$ in $\catdom$ maps to the restriction $\iota'\colon D\setminus\{0\}
\to D'\setminus\{0\}$. Note that $\iota'$ is a morphism since $\iota$ is injective and respects products.
Consequently,
the composition operation $D\mapsto\catf(D\setminus\{0\})$
describes a covariant functor $\catdom\to\catcat$.
\end{disc}

\section{Morphisms}\label{sec170415b}

In this section, we analyze morphisms in the category $\catf(A)$,
where $A$ is still a multiplicative, commutative, cancellative, pre-ordered monoid.
We begin with the morphisms in and out of the empty tuple $\emptytuple$.

\begin{prop}\label{empty_tuple_homset}
For each tuple $(x_n)$ in $\catf(A)$, we have 
\begin{align*}
\left|\mor{\mathcal{F}(A)}{\emptytuple}{(x_n)}\right| 
&= \begin{cases}
1 & \text{if $(x_n) = \emptytuple$}\\
0 & \text{otherwise}
\end{cases}
\\
\left|\mor{\mathcal{F}(A)}{(x_n)}{\emptytuple}\right| 
&\leq 1.
\intertext{If $A$ is a divisibility monoid, then}
\left|\mor{\mathcal{F}(A)}{(x_n)}{\emptytuple}\right| 
&= \begin{cases}
1 & \text{if $(x_n) = \emptytuple$ or each $x_n$ is invertible} \\
0 & \text{if $(x_n) \neq \emptytuple$ and some $x_n$ is non-invertible.}
\end{cases}
\end{align*}
\end{prop}

\begin{proof}
If $(x_n)\neq \emptytuple$, then 
there are no morphisms $\hat{f}\colon\emptytuple \to (x_n)$, since there are no functions  
$f\colon [N] \to \emptyset$. 
On the other hand, for $(x_n)=\emptytuple$, the unique function $f\colon\emptyset\to\emptyset$
vacuously determines the identity morphism $\id_{\emptytuple}\colon\emptytuple\to\emptytuple$.
This justifies our formula for $\left|\mor{\mathcal{F}(A)}{\emptytuple}{(x_n)}\right|$ from the statement of the proposition.

The morphisms  $\hat{f}\colon(x_n) \to \emptytuple$ correspond
exactly  to the functions  $f\colon\emptyset \to [N]$ such that  $x_n \leq 1_A$ for all $n\in[N]$. 
In particular, there is at most one such morphism $\hat f$.
This justifies our bound for $\left|\mor{\mathcal{F}(A)}{(x_n)}{\emptytuple}\right|$ from the statement.

In light of what we have shown, we assume for the rest of this proof that $A$ is a divisibility monoid and $(x_n)\neq\emptytuple$.
Since $A$ is a divisibility monoid, the unique function $f\colon\emptyset \to [N]$ satisfies  $x_n \leq 1_A$ for all $n$
if and only if 
$x_n \mid 1_A$ for all $n$, that is, if and only if 
each $x_n$ is invertible.
This gives the desired computation for 
$\left|\mor{\mathcal{F}(A)}{(x_n)}{\emptytuple}\right|$.
\end{proof}

The next example shows  that the divisibility monoid assumption is crucial for the last formula in Proposition~\ref{empty_tuple_homset}.
It applies similarly to subsequent results of this section, though we leave the details to the reader.

\begin{ex}\label{ex170421a}
Consider the real interval $(0,1]$ under multiplication with the standard ordering $\leq$ of real numbers. 
Then $a\leq 1$ for all $a\in (0,1]$, but the only invertible element of $(0,1]$ is $1$. 
Thus, the proof of 
Proposition~\ref{empty_tuple_homset}
shows that 
$\left|\mor{\mathcal{F}(A)}{(x_n)}{\emptytuple}\right|=1$ for all $(x_n)$ regardless of the invertibility of the $x_n$.
\end{ex}

The proof of the next result is similar to
the previous one. The point is that there is a unique function
$[N]\to[1]$ and there are $N$ functions $[1]\to[N]$.

\begin{prop}\label{singletons_unique}
If $(y)$ is a $1$-tuple in $\mathcal{F}(A)$ and $(x_n)$ is an $N$-tuple in $\mathcal{F}(A)$, then
\begin{align}
\left|\mor{\mathcal{F}(A)}{(y)}{(x_n)}\right| 
\label{eq170421b}
&=\begin{cases}
1&\text{if $y\leq x_1\cdots x_N$}\\
0&\text{otherwise}
\end{cases}
\\
\left|\mor{\mathcal{F}(A)}{(x_n)}{(y)}\right| 
\notag
&\leq |\{n\in[N]\mid x_n\leq y\}|
\leq N.
\\
\intertext{If in addition $A$ is a divisibility monoid, then}
\left|\mor{\mathcal{F}(A)}{(x_n)}{(y)}\right| 
\notag
&=|\{n_0\in[N]\mid \text{$x_n$ is invertible for all $n\neq n_0$ and $x_{n_0}\mid y$}\}|.
\end{align}
\end{prop}

Note that the next result applies to diagrams of other shapes that can be described in terms of triangular diagrams.

\begin{cor}\label{cor190115a}
Let morphisms $(y)\xra{\hat f}(x_n)$ and $(y)\xra{\hat g}(z_p)\xra{\hat h}(x_n)$
in $\catf(A)$ be given where $(y)$ is a 1-tuple. Then the following diagram commutes in $\catf(A)$.
$$\xymatrix{(y)\ar[r]^-{\hat g}\ar[rd]_{\hat f}
&(z_p)\ar[d]^{\hat h}\\
&(x_n)}$$
\end{cor}

\begin{proof}
Formula~\eqref{eq170421b} from Proposition~\ref{singletons_unique} implies that there is at most one morphism
$(y)\to(x_n)$. Our assumptions imply that there exists such a morphism, so all such morphisms must be equal.
\end{proof}

Next, we 
show
that $\catf(A)$ can have several initial objects.
Of course, they must be isomorphic, so this gives the first hint of our characterization of isomorphisms
in Theorem~\ref{iso-thm} 
below.

\begin{cor}\label{initial_object_thm}
Assume that $A$ is a divisibility monoid,
and let $y\in A$. Then $y$ is invertible in $A$ if and only if 
the 1-tuple $(y)$ is an initial object in $\mathcal{F}(A)$.
\end{cor}

\begin{proof}
The element $y \in A$ is invertible if and only if $y\mid 1$ if and only if $y \mid x$ for all $x \in A$, so the result follows directly from 
formula~\eqref{eq170421b}.
\end{proof}

\begin{prop}
Assume that $A$ is a divisibility monoid.
If $A$ has a non-invertible element, then
$\mathcal{F}(A)$ has no terminal object.
\end{prop}

\begin{proof}
We argue by contrapositive. Assume
that $(z_p)$ is
terminal in $\catf(A)$, and set $z = \prod_p z_p \in A$.
Then every tuple $(x_n)$ has 
a
morphism of the form $(x_n) \to (z_p)$. 
In particular, for each $\alpha \in A$, there exists a morphism of the form $(\alpha) \to (z_p)$,
hence $\alpha \mid z$. 
In particular, $z^2 \mid z$ so $z = az^2$ for some $a\in A$. Cancellativity implies that $1 = az$, so $z$ is invertible. 
Thus, each element 
$\alpha\in A$
must divide the invertible element $z$, and it follows that each element  
$\alpha\in A$
is also invertible, 
as desired.
\end{proof}

In general, 
the morphisms $(x_n)\to(y_m)$ in $\catf(A)$
correspond to functions $[M]\to[N]$,
so we have $|\mor{\catf(A)}{(x_n)}{(y_m)}|\leq N^M$.
As we see next, this bound is not sharp in general, though it can be
for specific examples.

\begin{ex}\label{ex170421b}
Consider the divisibility monoid $\bbz\setminus 0$, and let $M,N\in\bbn$.
There are no morphisms from the $N$-tuple $(2,2,\ldots,2)$ to the $M$-tuple $(3,3,\ldots,3)$. 
On the other extreme, there are $N^M$ morphisms from the $N$-tuple $(1,1,\ldots,1)$ to the $M$-tuple $(3,3,\ldots,3)$. 
Most tuples fit somewhere in the middle of this spectrum; for instance, there are exactly two morphisms $(1,2)\to(1,2)$.
\end{ex}

Next, we characterize the isomorphisms in our 
category.

\begin{thm}\label{iso-thm}
Assume that $A$ is a divisibility monoid.
A morphism $\hat{f}\colon(x_n) \to (y_m)$ in $\catf(A)$ is an isomorphism if and only (1) $N = M$, 
(2) $f$ is a bijection, and 
(3) there is an $N$-tuple 
$(u_n)$ of invertible monoid elements
such that   
$u_nx_n = y_{f^{-1}(n)}$
for each $n\in[N]$.
\end{thm}

\begin{proof}
For the 
the forward implication,
let $\hat{f}\colon(x_n) \to (y_m)$ be an isomorphism in $\catf(A)$ with inverse morphism $\hat{g}\colon(y_m) \to (x_n)$.
It follows that $\widehat{\id_{[N]}}=\text{id}_{(x_n)}=\hat{g} \circ \hat{f}=\widehat{f \circ g}$, and similarly
$\widehat{\id_{[M]}}=\widehat{g \circ f}$. We conclude that $g \circ f = \text{id}_{[M]}$ and 
$f \circ g = \text{id}_{[N]}$, so $f$ and $g$ are inverse bijections. In particular, we have $N=M$, and the functions $f$ and $g$ are bijections.
Furthermore, it follows that the divisibility conditions making $\hat f$ and $\hat g$ morphisms in $\catf(A)$
read as 
$x_n \mid y_{f^{-1}(n)}$  and $y_{f^{-1}(n)} \mid x_{n}$, respectively, for each $n$. 
Since $A$ is cancellative, this implies that $u_nx_n=y_{f^{-1}(n)}$ for some invertible monoid element $u_n$, as desired.

For  
the converse,
given  $N$-tuples $(x_n)$ and  $(u_n)$ such that each $u_n$ is invertible, and 
given any permutation $f \in S_N$, it follows readily that the morphism
$\hat f\colon(x_n) \to (u_{f(n)} x_{f(n)})$ is a well-defined morphism in $\catf(A)$ with inverse $\hat{f}^{-1}=\widehat{f^{-1}}$.
\end{proof}

\begin{disc}\label{disc171015b}
Assume that $A$ is a divisibility monoid, and
consider a morphism $\hat{f}\colon(x_n) \to (y_m)$ 
in $\catf(A)$. Assume for this paragraph that $(y_m)\neq\emptytuple$.
Therefore, $(x_n)\neq\emptytuple$ as well by Proposition~\ref{empty_tuple_homset}.
The divisibility condition making $\hat f$ a morphism in $\catf(A)$ says that
$x_n\mid\prod_{m\in f^{-1}(n)}y_m$ for each $n\in[N]$, so there is an element $r_n\in A$ such that
$r_nx_n=\prod_{m\in f^{-1}(n)}y_m$.
In particular, the element $r:=\prod_nr_n\in A$ satisfies
$$\textstyle
r\prod_{n=1}^Nx_n=\prod_{n=1}^Nr_n\cdot\prod_{n=1}^Nx_n=\prod_{n=1}^Nr_nx_n=\prod_{n=1}^N\prod_{m\in f^{-1}(n)}y_m=\prod_{m=1}^My_m.$$
Moreover, the fact that $A$ is cancellative implies that the elements $r_n$ and $r$ are unique in $A$ with these properties.

In the case $(y_m)\neq\emptytuple$, we again have a unique $r\in A$ such that 
$r\prod_{n=1}^Nx_n=\prod_{m=1}^My_m=1$; moreover, $r$ is invertible in $A$ in this case. Indeed, Proposition~\ref{empty_tuple_homset}
implies that $\prod_{n=1}^Nx_n$ is invertible, so $r=(\prod_{n=1}^Nx_n)^{-1}$ does the job.
\end{disc}

Next, we formalize a decomposition result for morphisms suggested in Example~\ref{basic-example}.
In light of Proposition~\ref{empty_tuple_homset}, the non-empty assumption on $(x_n)$ and $(y_m)$
is not much of a restriction.

\begin{defn}\label{defn170715b}
Assume that $A$ is a divisibility monoid.
A morphism $\hat f\colon(x_n)\to (y_m)$  in $\catf(A)$ between non-empty tuples
is a \emph{factorization morphism} provided that the underlying function $f\colon[M]\to[N]$ is surjective
and for all $n\in [N]$ we have $x_n=\prod_{m\in f^{-1}(n)}y_n$.
A \emph{divisibility morphism} is any morphism $\widehat{\id_{[P]}}\colon (z_p)\to (a_pz_p)$ with $P\geq 1$ and $(a_p)\in\catf(A)$.
(See~\ref{basic-example} for specific examples.)
\end{defn}

\begin{prop}\label{prop170423a}
Let $A$ be a divisibility monoid.
Let $\hat f\colon(x_n)\to (y_m)$ be a morphism in $\catf(A)$ between non-empty tuples.
Then $\hat f$ decomposes as the composition 
$$(x_n)\xra{\hat\epsilon}(z_p)\xra{\widehat{\id_{[P]}}}(a_pz_p)
\xra{\hat\phi}(y_m)$$
where we have the following:
\begin{enumerate}[\rm(1)]
\item $\epsilon$ is injective, 
and 
$\hat\epsilon\colon(x_n)\to(z_p)$
corresponds to dropping invertible elements and permuting the non-dropped elements;
\item $P=|\im(f)|$, and $\widehat{\id_{[P]}}\colon(z_p)\to(a_pz_p)$ is a divisibility morphism; and
\item $\phi$ is surjective, and $\hat\phi\colon(a_pz_p)\to(y_m)$ is a factorization morphism.
\end{enumerate}
Following this notation, we have $\prod_my_m=\left(u\prod_pa_p\right)\left(\prod_nx_n\right)$ where $u$ is the product of the
invertible elements dropped by $\hat\epsilon$.
That is, in the notation of Remark~\ref{disc171015b}, we have $r=u\prod_pa_p$.
\end{prop}

Before proving this result, we present an example illustrating the proof.

\begin{ex}\label{ex170423a}
We consider the divisibility monoid $\bbz\setminus\{0\}$
and the morphism $(6,1,35)\xra{\hat f}(2,7,33,65)$ corresponding to the divisibility conditions
$6\mid 2\cdot 33$ and $35\mid 7\cdot 65$, along with the fact that $1$ divides the empty product.
Note that $\hat f$ drops the invertible element $1$ in $(6,1,35)$, so we start our factorization as
$$(6,1,35)\xra{\hat\epsilon}(6,35)\xra{\hat\phi^{(6,35)}_{(2,7,33,65)}}(2,7,33,65)$$
where $\hat\epsilon$ drops the invertible element, and $\hat\phi^{(6,35)}_{(2,7,33,65)}$ corresponds to the divisibility conditions
$6\mid 2\cdot 33$ and $35\mid 7\cdot 65$.
Recall the superscript/subscript notation for $\hat\phi$ from Definition~\ref{defn180217a}.

Next, we write these divisibility conditions as $11\cdot 6=2\cdot 33$ and $13\cdot 35=7\cdot 65$, 
allowing us to rewrite $\hat f$ in the form $\hat\phi^{(6\cdot 11,35\cdot 13)}_{(2,7,33,65)}\circ\widehat{\id_{[2]}}\circ\hat\epsilon$ as follows
$$(6,1,35)\xra{\hat\epsilon}(6,35)\xra{\widehat{\id_{[2]}}}(6\cdot 11,35\cdot 13)\xra{\hat\phi^{(6\cdot 11,35\cdot 13)}_{(2,7,33,65)}}(2,7,33,65)$$
where $\widehat{\id_{[2]}}$ corresponds to the divisibility conditions $6\mid 6\cdot 11$ and $35\mid 35\cdot 13$,
and $\hat\phi^{(6\cdot 11,35\cdot 13)}_{(2,7,33,65)}$ corresponds to the factorization $6\cdot 11=2\cdot 33$ and $35\cdot 13=7\cdot 65$.

Formally, we have $P=2$ and
\begin{align*}
[4]\xra{f}[3]
&&1\mapsto 1
&&2\mapsto 3
&&3\mapsto 1
&&4\mapsto 3
\\
[2]\xra{\epsilon}[3]
&&1\mapsto 1
&&2\mapsto 3
\\
[4]\xra{\phi}[2]
&&1\mapsto 1
&&2\mapsto 2
&&3\mapsto 1
&&4\mapsto 2
\end{align*}
\end{ex}

\begin{proof}[Proof of Proposition~\ref{prop170423a}]
Let $\im(f)=\{n_1,\ldots,n_P\}$ where $n_1<\cdots<n_P$.
Define $\epsilon\colon[P]\to[N]$ by the formula $\epsilon(p)=n_p$.
Note that for each $n\notin\im(f)$, the divisibility condition for $x_n$ by definition says that $x_n$ divides the empty product
so $x_n$ is invertible. 
It follows readily that $\hat\epsilon\colon(x_n)\to(x_{n_1},\ldots,x_{n_P})$ is a morphism in $\catf(A)$
that drops the invertible elements $x_n$ for all $n\notin\im(f)$.

Define $\phi\colon[M]\to[P]$ by the formula $\phi(m)=p$ where $p$ is the unique element of $[P]$ such that $f(m)=n_p$.
By construction, we have $\phi^{-1}(p)=f^{-1}(n_p)$.
Thus, the divisibility condition $x_{n_p}\mid\prod_{m\in f^{-1}(n_p)}y_m$ reads as
$x_{n_p}\mid\prod_{m\in\phi^{-1}(p)}y_m$.
We conclude that $\hat\phi^{(x_{n_p})}_{(y_m)}\colon(x_{n_p})=(x_{n_1},\ldots,x_{n_P})\to(y_m)$ is a morphism in $\catf(A)$.
Furthermore, by construction we have $\epsilon\circ\phi=f$ since $\epsilon(\phi(m))=\epsilon(p)=n_p=f(m)$,
so we have $\hat f=\widehat{\epsilon\circ\phi}=\hat\phi^{(x_{n_p})}_{(y_m)}\circ\hat\epsilon$.

For $p=1,\ldots,P$ set $z_p=x_{n_p}$.
Thus, we have $\hat\epsilon\colon(x_n)\to(z_p)=(x_{n_p})$ and
$\hat\phi^{(z_{p})}_{(y_m)}\colon(z_p)\to(y_m)$.
The divisibility relation $z_p=x_{n_p}\mid\prod_{m\in f^{-1}(n_p)}y_m$ implies that there is an element $a_p\in A$ such that
$a_pz_p=\prod_{m\in f^{-1}(n_p)}y_m$.
Then $\widehat{\id_{[P]}}\colon(z_p)\to(a_pz_p)$ is a morphism (a divisibility morphism) in $\catf(A)$.
Using the same $\phi$ as above, the equality $a_pz_p=\prod_{m\in f^{-1}(n_p)}y_m$ shows that
$\hat\phi^{(a_pz_p)}_{(y_m)}\colon(a_pz_p)\to(y_m)$ is a morphism (a factorization morphism) in $\catf(A)$.
Furthermore, the equality $\id_{[P]}\circ\phi=\phi$ implies that the morphism $\hat\phi^{(x_{n_p})}_{(y_m)}$
factors as $\hat\phi^{(x_{n_p})}_{(y_m)}=\widehat{\id_{[P]}\circ\phi}=\hat\phi^{(a_pz_p)}_{(y_m)}\circ\widehat{\id_{[P]}}$.
Thus, in light of the equality $\hat f=\hat\phi^{(x_{n_p})}_{(y_m)}\circ\hat\epsilon$ from the previous paragraph,
we have the factorization
$\hat f=\hat\phi^{(x_{n_p})}_{(y_m)}\circ\hat\epsilon=\hat\phi^{(a_pz_p)}_{(y_m)}\circ\widehat{\id_{[P]}}\circ\hat\epsilon$, as desired.

Lastly, we compute 
\begin{align*}
\textstyle\prod_my_m
&\textstyle=\prod_pa_pz_p
=\left(\prod_pa_p\right)\left(\prod_pz_p\right)
=\left(\prod_pa_p\right)u\left(\prod_nx_n\right)
\end{align*}
where the steps in this sequence come from the construction of the morphisms $\hat\phi^{(a_pz_p)}_{(y_m)}$, $\widehat{\id_{[P]}}$,
and $\hat\epsilon$, respectively.
\end{proof}

The forward implications in the next result are fairly routine.
However, the converses are a bit surprising, the fact that 
the order-constrained maps
are rich enough to be able to see injectivity and surjectivity of the underlying 
functions.

\begin{thm}
\label{epic_monic_theorem}
Assume that $A$ is a divisibility monoid, and let $\hat{f}\colon(x_n) \to (y_m)$ be a morphism in $\mathcal{F}(A)$. 
\begin{enumerate}[\rm(a)]
\item\label{item170423a}
The function $f$ is injective if and only if the morphism $\hat{f}$ is epic. 
\item\label{item170423b}
The function $f$ is surjective if and only if the morphism $\hat{f}$ is monic. 
\end{enumerate}
\end{thm}

\begin{proof}
\eqref{item170423a} For the forward implication, 
if $f$ is injective then $f$ cancels on the left as a function. Hence, if 
$\hat{g}$ and $\hat{h}$ are morphisms in $\mathcal{F}(A)$ such that
$\hat{g} \circ \hat{f} = \hat{h}\circ \hat{f}$, then 
$\widehat{f \circ g}= \widehat{f \circ h}$, so  $f \circ g = f \circ h$; the injectivity of $f$ implies that $g = h$, so $\hat{g} = \hat{h}$.

For the converse, we assume that $f$ is not injective (in particular, $M\geq 2$) and show that $\hat f$ is not epic. 
Permuting the entries of the tuple $(y_m)$
corresponds to composition of $\hat f$ with an isomorphism, hence this does not affect the question of whether 
$\hat f$ is epic, nor does it change the fact that $f$ is not injective. 
Thus, we assume without loss of generality that $f(M-1)=f(M)$.

To show that $\hat f$ is not epic, we explicitly construct two morphisms $\hat g,\hat h\colon(y_m)\to(y_1,\ldots,y_M,1)$ such that
$\hat g\neq\hat h$ but $\hat g\circ\hat f=\hat h\circ\hat f$.
The morphism $\hat g$ corresponds to the divisibility conditions $y_m\mid y_m$ for all $m<M$ and $y_M\mid 1\cdot y_M$.
In other words, we have $g\colon[M+1]\to [M]$ defined by $g(m)=m$ for all $m\leq M$ and $g(M+1)=M$.
The morphism $\hat h$ corresponds to the divisibility conditions $y_m\mid y_m$ for all $m<M-1$ 
along with $y_{M-1}\mid 1\cdot y_{M-1}$ and $y_M\mid y_M$.
In other words, we have $h\colon[M+1]\to [M]$ defined by $h(m)=m$ for all $m\leq M$ and $h(M+1)=M-1$.

Since $g(M+1)=M\neq M-1=h(M+1)$, we have $g\neq h$ hence $\hat g\neq\hat h$.
On the other hand, we have $f\circ g=f\circ h$. Indeed, for $m\leq M$ we have
$f(g(m))=f(m)=f(h(m))$.
Also, $f(g(M+1))=f(M)=f(M-1)=f(h(M+1))$, so we have $f\circ g=f\circ h$, and thus
$\hat g\circ\hat f=\widehat{f\circ g}=\widehat{f\circ h}
=\hat h\circ\hat f$,
as desired.

\eqref{item170423b}
The forward implication  is established like the forward implication of part~\eqref{item170423a}.
For the converse, we assume that $f$ is not surjective and show that $\hat f$ is not monic. 
Permuting the entries of the tuple $(x_n)$, we assume without loss of generality that $N\notin\im(f)$.
In particular, the divisibility condition for $\hat f$ implies that $x_N$ divides the empty product, so $x_N$ is invertible.

In order to show that $\hat f$ is not monic, we explicitly construct two morphisms $\hat g,\hat h\colon(x_1,\ldots,x_{N},x_{N})\to(x_n)$ such that
$\hat g\neq\hat h$ but $\hat f\circ\hat g=\hat f\circ\hat h$.
The morphism $\hat g$ corresponds to the divisibility conditions $x_n\mid x_n$ for all $n\leq N$ as well as 
the second copy of $x_{N}$ in $(x_1,\ldots,x_{N},x_{N})$ dividing the empty product.
In other words, we have $g\colon[N]\to [N+1]$ defined by $g(n)=n$ for all $n\leq N$.
The morphism $\hat h$ corresponds to the divisibility conditions $x_n\mid x_n$ for  $n<N$ as well as
the first copy of $x_{N}$ in $(x_1,\ldots,x_{N},x_{N})$ dividing the empty product
and the second copy of $x_N$ dividing $x_N$.
In other words, we have $h\colon[N]\to [N+1]$ defined by $h(n)=n$ for all $n< N$ and $h(N)=N+1$.

As in part~\eqref{item170423a}, it is straightforward to show that $g\neq h$ and $g\circ f=h\circ f$, so 
$\hat g\neq\hat h$ and $\hat f\circ\hat g=\hat f\circ\hat h$,
as desired.
\end{proof}

\begin{ex}
With Theorem~\ref{epic_monic_theorem} in hand, it is straightforward to 
say which of the morphisms from Example~\ref{basic-example} are monic and which ones are epic,
when $A$ is a divisibility monoid.
For instance, morphisms $(x_n)\to(y_m)$ that just drop invertible elements correspond to injective functions $[M]\to[N]$,
so these morphisms are epic.
For factorization morphisms, the underlying functions are surjective, so these morphisms are monic.
And the functions underlying divisibility morphisms are bijections, so these morphisms are both epic and monic.
In particular, in the decomposition $\hat f=\hat\phi\circ\widehat{\id_{[P]}}\circ\hat\epsilon$ from Proposition~\ref{prop170423a},
the morphisms $\hat\epsilon$ and $\widehat{\id_{[P]}}$ are epic, while $\widehat{\id_{[P]}}$ and $\hat\phi$ are monic.

Note that the divisibility morphisms therefore give plentiful examples of morphisms in $\catf(A)$ that are both monic and
epic but are not isomorphisms. For instance, in $\catf(\bbz\setminus\{0\})$ the divisibility morphism $(2)\to(6)$
is monic and epic, but not an isomorphism,
by
Theorem~\ref{iso-thm}.
\label{ex:retracts}
\end{ex}

\section{Functors
and Symmetric Monoidal Structure}
\label{sec170415a}

In this section, we investigate several natural functors involving $\catf(A)$,
and we show in Theorem~\ref{prop190106b} that $\catf(A)$ is a symmetric and strict monoidal category,
where $A$ is still a multiplicative, commutative, cancellative, pre-ordered monoid.

\begin{defn}\label{presheaf}
Let $\mathfrak{F}\colon\mathcal{F}(A) \to \catset$ be the contravariant functor defined by mapping an $N$-tuple 
$(x_n)$ to the set $[N]$, and by mapping each morphism $\hat{f}\colon(x_n) \to (y_m)$ to the underlying function 
$f\colon[M] \to [N]$.
\end{defn}

\begin{disc}\label{disc170430a}
It is straightforward to show that $\mathfrak F\colon\mathcal{F}(A) \to \catset$ is indeed a contravariant functor. 
Moreover, it is faithful because
$\mor{\catf(A)}{(x_n)}{(y_m)}$ is defined as a subset of $\mor{\catset}{[M]}{[N]}$.

On the other hand, $\mathfrak F$ is not full in general. 
For instance, 
if $A$ is a divisibility monoid with at least one non-invertible element $y$, then
$\mor{\catf(A)}{(y)}{(1)}=\emptyset\neq\mor{\catset}{[1]}{[1]}$,
so $\mathfrak F$ 
is not
full in this case.

Similarly, $\mathfrak F$ is not representable in general. 
Indeed, continue with 
a divisibility monoid $A$ with at least one non-invertible element $y$, and suppose that
$(z_p)\in\catf(A)$ is such that $[N]=\mathfrak F((x_n))\cong\mor{\catf(A)}{(x_n)}{(z_p)}$ for all $(x_n)$ in
$\catf(A)$. 
In particular, for each $N$-tuple $(x_n)$ in
$\catf(A)$ we must have $|\mor{\catf(A)}{(x_n)}{(z_p)}|=N$.
However, for the $N$-tuple $(1,\ldots,1)$ we have
$|\mor{\catf(A)}{(1,\ldots,1)}{(z_p)}|=|\mor{\catset}{[P]}{[N]}|=N^P$. 
The equation $N=N^P$ implies that $P=1$, so $(z_p)=(z_1)$.
For an arbitrary $1$-tuple $(x)$, the condition
$|\mor{\catf(A)}{(x)}{(z_1)}|=1$ implies that $x\mid z_1$ for all $x\in A$;
for example, we then have $yz_1\mid z_1$, contradicting the assumption that $y$ be non-invertible.
\end{disc}

Next, we consider two functors between $A$ and 
$\catf(A)$.

\begin{disc}\label{disc171015a}
Consider the monoid $A$ as a small category with object set $A$ and morphisms
given by the inequalities in $A$:
\begin{equation*}
\mor{A}xy=
\begin{cases}
\{x\leq y\}&\text{if $x\leq y$}\\
\emptyset&\text{otherwise.}
\end{cases}
\end{equation*}
In other words, $A$ is a small category where every
morphism-set has at most one element. 
With this convention, 
when we feel it is helpful to consider things functorially,
we consider the inclusion functor $\mathcal I\colon\catpom\to\catcat$,
so, for each pre-ordered monoid $B$, we denote by $\mathcal I(B)$ the corresponding small
category where every
morphism-set has at most one element.

As in Corollary~\ref{cor190115a}, the fact that every morphism set in $A$ has at most one element
implies that any diagram
$$\xymatrix{a\ar[r]\ar[rd]&b\ar[d]\\ &c}$$
in $A$ automatically commutes.
\end{disc}

\begin{defn}\label{defn170501b}
Let $\mathfrak{A}_A\colon\mathcal{F}(A) \to A$ be the covariant functor defined
on objects by mapping an $N$-tuple $(x_n)$ to the product
$\prod_n x_n\in A$. 
For each morphism $\hat{f}\colon (x_n) \to (y_m)$, let $\mathfrak{A}_A(\hat f)$ be the pre-order relation $\prod_n x_n \leq \prod_m y_m$
in 
$A$,
which is satisfied because $f$ is order-constrained.
\end{defn}

\begin{disc}\label{disc170430b}
It is straightforward to show that $\mathfrak A_A$ is indeed a covariant functor. 
Unlike our previous functor, though, $\mathfrak A_A$ is neither full nor faithful in general. 
Indeed, consider the divisibility monoid $A=\bbz\setminus\{0\}$.
It is straightforward to show that
$\mor{\catf(\bbz\setminus\{0\})}{(4,9)}{(6,6)}=\emptyset\neq\mor{\bbz\setminus\{0\}}{36}{36}$, so $\mathfrak A_A$ is not full.
Similarly, we have
$|\mor{\catf(\bbz\setminus\{0\})}{(2,2)}{(2,2)}|=2>1=|\mor{\bbz\setminus\{0\}}{4}{4}|$, so $\mathfrak A_A$ is not faithful.
\end{disc}

\begin{defn}\label{defn170501c}
Let $\mathfrak{B}_A\colon A\to\mathcal{F}(A)$ be the covariant functor defined by mapping an element
$x\in A$ to the 1-tuple $(x)$,
and by mapping each order relation $x\leq y$ to the morphism $\widehat{\id_{[1]}}\colon (x)\to(y)$
which is a morphism in $\catf(A)$ by definition.
\end{defn}

\begin{disc}\label{disc170430c}
It is straightforward to show that $\mathfrak B_A$ is indeed a covariant functor. 
Unlike our previous functor, though, we have the following.
\end{disc}

\begin{prop}\label{prop170430a}
The functor $\mathfrak B_A\colon A\to\mathcal{F}(A)$ is  fully faithful.
\end{prop}

\begin{proof}
This follows directly from 
formula~\eqref{eq170421b} from Proposition~\ref{singletons_unique}.
\end{proof}

\begin{prop}\label{prop170430b}
The functors $\mathfrak A_A$ and $\mathfrak B_A$ are adjoints, that is, for all $y\in A$ and $(x_n)\in\catf(A)$ we have
natural bijections
$$\mor{A}{y}{\mathfrak A_A((x_n))}
\cong\mor{\catf(A)}{\mathfrak B_A(y)}{(x_n)}.$$
In addition, the composition $\mathfrak{A}_A\mathfrak{B}_A\colon A\to A$ is the identity.
\end{prop}

\begin{proof}
To show that the composition $\mathfrak{A}_A\mathfrak{B}_A$ is the identity on objects, we compute directly:
$\mathfrak{A}_A(\mathfrak{B}_A(y))
=\mathfrak{A}_A((y))=y$. 
To show that $\mathfrak{A}_A\mathfrak{B}_A$ is the identity on morphisms, 
let $\phi\colon y\to z$ be a morphism in $A$.
That is, we have $y\leq z$ in $A$ and the arrow $\phi$ is the unique arrow in $\mor Ayz$.
The composition $\mathfrak{A}_A\mathfrak{B}_A$ maps arrows in $\mor Ayz$ to arrows in $\mor Ayz$,
so $\mathfrak{A}_A(\mathfrak{B}_A(\phi))$ is the unique arrow in $\mor Ayz$,
i.e., we have $\mathfrak{A}_A(\mathfrak{B}_A(\phi))=\phi$, as desired.

The adjoint 
bijection
follows from the next display wherein the equalities
are by definition
and the  
bijection
is from equation~\eqref{eq170421b} in Proposition~\ref{singletons_unique}.\begin{align*}
\mor{A}{y}{\mathfrak A_A((x_n))}
&\textstyle
=
\mor{A}{y}{\prod_nx_n}\\
&
\cong
\mor{\catf(A)}{(y)}{(x_n)}\\
&
=
\mor{\catf(A)}{\mathfrak B_A(y)}{(x_n)}
\end{align*}
The naturality of the bijection here follows as in the preceding paragraph, using the fact that these morphism
sets in $A$ and $\catf(A)$ have at most one element each.
\end{proof}

Our notations for $\mathfrak A_A$ and $\mathfrak B_A$  are intentionally suggestive of the next result.

\begin{prop}\label{prop170430c}
Consider the inclusion functor $\mathcal I\colon\catpom\to\catcat$ 
of
Remark~\ref{disc171015a}
and the category of factorization functor $\catf\colon\catpom\to\catcat$.
Then the operations $\mathfrak A\colon\catf\to\mathcal I$ and $\mathfrak B\colon\mathcal I\to\catf$ are 
natural transformations. 
\end{prop}

\begin{proof}
First, observe that for each pre-ordered monoid $A$, we have $\mathcal I(A)=A$, and the operations
$\mathfrak{A}_A\colon\mathcal{F}(A) \to \mathcal I(A)$ and
$\mathfrak{B}_A\colon \mathcal I(A)\to\mathcal{F}(A)$ are covariant functors, that is, they are morphisms in $\catcat$.
Furthermore, for each morphism $\phi\colon A\to B$ in $\catpom$, we see next that
the following diagram commutes
\begin{gather*}
\xymatrix{
A\ar[r]^-{\mathfrak B_A}
\ar[d]_{\phi}
&\catf(A)
\ar[r]^-{\mathfrak A_A}
\ar[d]_{\catf(\phi)}
&A
\ar[d]_{\phi}
\\
B\ar[r]_-{\mathfrak B_B}
&\catf(B)
\ar[r]_-{\mathfrak A_B}
&B
}\\
\xymatrix{
y\ar@{|->}[r]\ar@{|->}[d]
&(y)\ar@{|->}[d]
&(x_n)\ar@{|->}[r]\ar@{|->}[d]
&\textstyle\prod_{n}x_n\ar@{|->}[d]
\\
\phi(y)\ar@{|->}[r]
&(\phi(y))
&(\phi(x_n))\ar@{|->}[r]
&\textstyle\prod_{n}\phi(x_n)
}
\end{gather*}
so we indeed have natural transformations.
\end{proof}

We use the following 
concatenation tensor product in our  treatment of weak divisibility and 
weak primeness in Sections~\ref{sec170708a}, \ref{sec170702a}, and~\ref{accp-implies-atomic}
below. 
Proposition~\ref{prop170626a} and Theorem~\ref{prop190106b} document some of its properties.

\begin{defn}\label{defn170605a}
Fix non-empty tuples $(x_n)$ and $(y_m)$ in $\catf(A)$, and define  
their concatenation tensor product
as
$$(x_n)\otimes(y_m)=(x_1,\ldots,x_N,y_1,\ldots,y_M).$$
We also define $(x_n)\otimes\emptytuple=(x_n)=\emptytuple\otimes(x_n)$ and
$\emptytuple\otimes\emptytuple=\emptytuple$.
(In other words, the tensor product here is the binary operation on the free monoid on $A$.)
Note that this makes the length of the tensor product of two tuples equal to the sum of the lengths of the factors.
Note further that the tuples $(x_n)\otimes(y_m)$ and $(y_m)\otimes(x_n)$ are isomorphic in $\catf(A)$,
by permuting the entries, and similarly for tensor products involving $\emptytuple$.

We now make 
this into a bifunctor $\catf(A)\times\catf(A)\to\catf(A)$.
Let $\hat f\colon (x_n)\to (w_q)$ and $\hat g\colon (y_m)\to (z_p)$ be morphisms in $\catf(A)$ between non-empty tuples.
Intuitively, the induced morphism 
$$(x_n)\otimes(y_m)=(x_1,\ldots,x_N,y_1,\ldots,y_M)\to(w_1,\ldots,w_Q,z_1,\ldots,z_P)=(w_q)\otimes(z_p)$$
is given by the  
conditions  $x_n\leq\prod_{q\in f^{-1}(n)}w_q$ and 
$y_m\leq\prod_{p\in g^{-1}(m)}z_p$ coming from the fact that $\hat f$ and $\hat g$ are morphisms in $\catf(A)$.
Formally, the induced morphism $\hat f\otimes\hat g\colon(x_n)\otimes(y_m)\to (w_q)\otimes(z_p)$ is defined to be
$\widehat{f*g}$ where $f*g\colon[Q+P]\to[N+M]$ is defined as
$(f*g)(q)=f(q)$ for $q\in [Q]$, and $(f*g)(Q+p)=N+g(p)$ for $p\in[P]$.
Following the intuitive description of $\hat f\otimes\hat g$ above, one checks readily that the fact that $\hat f$ and $\hat g$ are morphisms in $\catf(A)$
implies that $\hat f\otimes\hat g=\widehat{f*g}$ is also a morphism.
Note that this formal definition also applies to morphisms with domain or codomain $\emptytuple$, using $M=0$, $P=0$, $N=0$, or $Q=0$,
as necessary.
It is straightforward to show that this gives $\emptytuple\otimes\hat f=\hat f=\hat f\otimes \emptytuple$.
\end{defn}

\begin{prop}\label{prop170626a}
The above tensor product is a bifunctor $\catf(A)\times\catf(A)\to\catf(A)$ such that
$\emptytuple\otimes-=-\otimes\emptytuple$ is the identity functor.  
\end{prop}

\begin{proof}
The operation $\emptytuple\otimes-=-\otimes\emptytuple$ is the identity functor, by definition.
To see that $\otimes$ respects identity morphisms $\widehat{\id_{[N]}}\colon(x_n)\to(x_n)$ and $\widehat{\id_{[M]}}\colon(y_m)\to(y_m)$,
notice that $\id_{[N]}*\id_{[M]}=\id_{[N+M]}\colon[N+M]\to[N+M]$, and so 
$$\id_{(x_n)}\otimes\id_{(y_m)}=\widehat{\id_{[N]}}\otimes\widehat{\id_{[M]}}=\widehat{\id_{[N+M]}}=\id_{(x_n)\otimes(y_m)}.$$
To check that this respects compositions, consider morphisms
$(x_n)\xra{\hat f}(w_q)\xra{\hat h}(u_j)$ and $(y_m)\xra{\hat g}(z_p)\xra{\hat k}(v_i)$.
Then $\hat f\otimes\hat g=\widehat{f*g}$ and $\hat h\otimes\hat k=\widehat{h*k}$.
We show that $(f*g)\circ(h*k)=(f\circ h)*(g\circ k)$.
To this end, for $j\in [J]$ the first equality in the next display is by definition,
as is the second one using the condition $h(j)\in [Q]$, and similarly for the other equalities.
\begin{align*}
(f*g)((h*k)(j))
&=(f*g)(h(j))
=f(h(j))\\
&=(f\circ h)(j)
=((f\circ h)*(g\circ k))(j)
\\
(f*g)((h*k)(J+i))
&=(f*g)(Q+k(i))
=N+g(k(i))\\
&=N+(g\circ k)(i)
=((f\circ h)*(g\circ k))(J+i)
\end{align*}
This explains the third equality in the next display
\begin{align*}
(\hat h\otimes\hat k)\circ(\hat f\otimes\hat g)
&=\widehat{h*k}\circ\widehat{f*g}
=[(f*g)\circ (h*k)]^\wedge
=[(f\circ h)*(g\circ k)]^\wedge\\
&=\widehat{f\circ h}\otimes \widehat{g\circ k}
=(\hat h\circ\hat f)\otimes(\hat k\circ\hat g)
\end{align*}
and the other equalities are by definition.
\end{proof}

\begin{thm}\label{prop190106b}
With the above tensor product, $\catf(A)$ is a symmetric and strict monoidal category. 
\end{thm}

\begin{proof}
We have already seen that $\emptytuple\otimes-=-\otimes\emptytuple$ is the identity functor.
Also, it follows by definition that $(x_n)\otimes[(y_m)\otimes(z_p)]=[(x_n)\otimes(y_m)]\otimes(z_p)$.
From this, it follows readily that the following diagrams commute
$$\xymatrix{
[[(x_n)\otimes(y_m)]\otimes(z_p)]\otimes(w_q)
\ar[r]^-{\id\otimes \id}_-{=\id}\ar[dd]_\id
&[(x_n)\otimes[(y_m)\otimes(z_p)]]\otimes(w_q)\ar[d]^\id
\\
&(x_n)\otimes[[(y_m)\otimes(z_p)]\otimes(w_q)]\ar[d]^{\id\otimes \id=\id}
\\
[(x_n)\otimes(y_m)]\otimes[(z_p)\otimes(w_q)]\ar[r]^-\id
&(x_n)\otimes[(y_m)\otimes[(z_p)\otimes(w_q)]]
\\
[(x_n)\otimes\emptytuple]\otimes(z_p)\ar[r]^\id\ar[rd]_{\id\otimes \id=\id}
&(x_n)\otimes[\emptytuple\otimes(z_p)]\ar[d]^{\id\otimes \id=\id}
\\
&(x_n)\otimes(z_p)
}$$
so $\catf(A)$ is a strict monoidal category.

For the symmetric structure, we need to define commutativity isomorphisms
$B_{(x_n),(y_m)}\colon(x_n)\otimes(y_m)\to (y_m)\otimes(x_n)$ as in the next display
$$(x_n)\otimes(y_m)=(x_1,\ldots,x_N,y_1,\ldots,y_M)\to(y_1,\ldots,y_M,x_1,\ldots,x_N)=(y_m)\otimes(x_n).$$
Formally, this is the  morphism 
$B_{(x_n),(y_m)}=\widehat{\zeta_{M,N}}\colon(x_n)\otimes(y_m)\to (y_m)\otimes(x_n)$ 
where $\zeta_{M,N}\colon[M+N]\to[N+M]$
is defined to be $\zeta_{M,N}(m)=N+m$ and $\zeta_{M,N}(M+n)=n$ for $m\in[M]$ and $n\in[N]$. 
Note that this makes sense even when $N=0$ or $M=0$.
It is straightforward to show that 
$\widehat{\zeta_{M,N}}$ is a morphism in $\catf(A)$, 
and furthermore that $\zeta_{N,M}\circ\zeta_{M,N}=\id_{[M+N]}$ so
$$\widehat{\zeta_{M,N}}\circ\widehat{\zeta_{N,M}}=[\zeta_{N,M}\circ\zeta_{M,N}]^\wedge=\widehat{\id_{[M+N]}}=\id_{(y_m)\otimes(x_n)}.$$
As in the proof of Proposition~\ref{prop170626a},  the next diagram  commutes
$$\xymatrix{
[(N+M)+P]
&[N+(M+P)]\ar[l]_-\id
&[(M+P)+N]\ar[l]_-{\zeta_{M+P,N}}
\\
[(M+N)+P]\ar[u]^-{\zeta_{M,N}*\id}
&[M+(N+P)]\ar[l]_-\id
&[M+(P+N)]\ar[l]_-{\id*\zeta_{P,N}}\ar[u]_-\id
}$$
which yields the following commutative hexagonal diagram
$$\xymatrix{
[(x_n)\otimes(y_m)]\otimes(z_p)\ar[r]^-\id\ar[d]_{\widehat{\zeta_{M,N}}\otimes \id}^{=B\otimes\id}
&(x_n)\otimes[(y_m)\otimes(z_p)]\ar[r]^-{\widehat{\zeta_{M+P,N}}}_-{=B}
&[(y_m)\otimes(z_p)]\otimes(x_n)\ar[d]^-\id
\\
[(y_m)\otimes(x_n)]\otimes(z_p)\ar[r]^-\id
&(y_m)\otimes[(x_n)\otimes(z_p)]\ar[r]^-{\id\otimes\widehat{\zeta_{P,N}}}_-{=\id\otimes B}
&(y_m)\otimes[(z_p)\otimes(x_n)]
}$$
thus completing the proof.
\end{proof}

\section{Weak Equivalences}\label{model-category}

\begin{assumption}
In this section, $A$ is a divisibility monoid.
\end{assumption}

Here we identify an important set of morphisms in $\catf(A)$, denoted
$\catw(A)$.
We then prove in Theorem~\ref{thm170501a} that the associated 
category of fractions $\catf(A)[\catw(A)^{-1}]$
is (naturally isomorphic to) $A$.
See~\cite{MR1406095,MR0210125} for background on this construction.
Then we exhibit some general properties of the set $\catw(A)$, e.g., 
Propositions~\ref{wk_equiv} and~\ref{prop190106a} show that $\catw(A)$ is a collection of weak equivalences
and admits a calculus of fractions.

Our definition for the set $\catw(A)$
is motivated by the 
fact that, for elements $a,b\in A$, the natural morphisms $(1,ab)\to(ab)\to(a,b)$ are not isomorphisms,
even though they do not carry any new factorization information about $A$.

\begin{defn}\label{defn170501d}
Let $\catw(A)$ be the set of morphisms $\hat{f}\colon(x_n) \to (y_m)$ in $\mathcal{F}(A)$ 
such that either (1) we have $(y_m)=\emptytuple$, or (2) we have $(y_m)\neq\emptytuple$, hence $(x_n)\neq\emptytuple$, 
and for each $n \in [N]$,
the element $r_n\in A$ from Remark~\ref{disc171015b} satisfying $r_nx_n =  \prod_{m \in f^{-1}(n)} y_m$
is  invertible in $A$.
\end{defn}

\begin{disc}\label{disc170501a}
If 
and $\hat\epsilon\colon(x_n)\to(z_p)$
corresponds to dropping invertible elements from $(x_n)$ and permuting the non-dropped elements,
then $\hat\epsilon\in\catw(A)$ because $r_n=1$ for each $n\in\im(\epsilon)$ and $r_n=x_n^{-1}$ for each $n\in[N]\setminus\im(\epsilon)$.
Similarly, 
every factorization morphism
is in $\catw(A)$.
Each isomorphism
is in $\catw(A)$, and a divisibility morphism is in $\catw(A)$ if and only if it is an isomorphism, by Theorem~\ref{iso-thm}.
\end{disc}

Next we show some of how the morphisms in $\catw(A)$ can see factorization properties in $A$.
We revisit this theme in subsequent sections.

\begin{lem}\label{lem190110a} 
For an element $r \in A$, the following conditions are equivalent.
\begin{enumerate}[\rm(i)]
\item \label{lem190110a1}
$r$ is invertible in $A$,
\item \label{lem190110a5}
each morphism $\hat f\colon (x_n)\to (y_m)$ such that $r\cdot\prod_nx_n=\prod_my_m$ is in $\catw(A)$, and
\item \label{lem190110a0}
some morphism $\hat f\colon (x_n)\to (y_m)$ such that $r\cdot\prod_nx_n=\prod_my_m$ is in $\catw(A)$.
\end{enumerate}
\end{lem}

\begin{proof}
The implication $\eqref{lem190110a5}\implies\eqref{lem190110a0}$ is straightforward.

$\eqref{lem190110a1}\implies\eqref{lem190110a5}$
Assume that $r$ is invertible in $A$, and consider a morphism $\hat f\colon (x_n)\to (y_m)$ such that $r\cdot\prod_nx_n=\prod_my_m$.
If $(y_m)=\emptytuple$, then $\hat f\in\catw(A)$ by definition, so assume that $(y_m)\neq\emptytuple$, and therefore $(x_n)\neq\emptytuple$.
For all $n\in[N]$, let $r_n\in A$ be such that $r_nx_n =  \prod_{m \in f^{-1}(n)} y_m$.
Then $r=\prod_nr_n$ by Remark~\ref{disc171015b}.
Since $r$ is invertible, each $r_n$ must be invertible as well, so
$\hat f\in\catw(A)$ by definition.

$\eqref{lem190110a0}\implies\eqref{lem190110a1}$
Assume that some morphism $\hat f\colon (x_n)\to (y_m)$ such that $r\cdot\prod_nx_n=\prod_my_m$ is in $\catw(A)$.
If $(y_m)=\emptytuple$, then $r$ is invertible by Remark~\ref{disc171015b}.
So, assume that $(y_m)\neq\emptytuple$, hence $(x_n)\neq\emptytuple$.
The elements $r_n$ from Definition~\ref{defn170501d}  
therefore satisfy $r=\prod_nr_n$. The condition $\widehat{\text{id}_{[M]}}\in\catw(A)$ 
implies that each $r_n$ is invertible in $A$, therefore so is $r$, as desired.
\end{proof}

The
set $\mathcal W(A)$ has a convenient description
in terms of the factorization result in Proposition~\ref{prop170423a}, as we show next.

\begin{prop}\label{prop170503c}
Let $\hat f\colon(x_n)\to (y_m)$ be a morphism in $\catf(A)$ between non-empty tuples.
Consider the decomposition $\hat f=\hat\phi\circ\widehat{\id_{[P]}}\circ\hat\epsilon$ from Proposition~\ref{prop170423a}
$$(x_n)\xra{\hat\epsilon}(z_p)\xra{\widehat{\id_{[P]}}}(a_pz_p)
\xra{\hat\phi}(y_m).$$
Then 
the morphism $\hat f$ is in $\catw(A)$,
if and only if the divisibility morphism
$\widehat{\id_{[P]}}$ is an 
isomorphism (see Theorem~\ref{iso-thm}).
\end{prop}

\begin{proof}
By Remark~\ref{disc171015b}, let $r\in A$ be the unique element of $A$ such that $\prod_my_m=r\prod_nx_n$.
Proposition~\ref{prop170423a} implies that $r=u\cdot\prod_pa_p$
where $u$ is an invertible element of $A$.
It follows that 
$a_p$ is invertible if and only if $r$ is invertible.
Thus $\widehat{\id_{[P]}}$ is an 
isomorphism if and only if $\hat f\in\catw(A)$ by Theorem~\ref{iso-thm} and Lemma~\ref{lem190110a}.
\end{proof}

For the next
result of this section, we recall the following notion from~\cite{MR0210125}.

\begin{defn}
\label{defn170501a}
Let $\catc$ be a category with a collection of 
morphisms
$\catw$. 
The 
\emph{category of fractions}
of $(\catc,\catw)$, if it exists, is a category $\catc[\catw^{-1}]$
with a universal functor $\mathfrak K\colon\catc\to\catc[\catw^{-1}]$
such that
\begin{enumerate}[(1)]
\item \label{defn170501a1}
for each $f\in\catw$ the morphism $\mathfrak K(f)$ is an isomorphism in $\catc[\catw^{-1}]$,
and
\item \label{defn170501a2}
for each category $\catd$ and each functor $\mathfrak L\colon \catc\to\catd$ such that for each $f\in\catw$ the 
morphism $\mathfrak L(f)$ is an isomorphism in $\catd$,
there is a unique functor $\mathfrak L'\colon\catc[\catw^{-1}]
\to\catd$ making the following diagram commute
$$\xymatrix{
\catc\ar[r]^-{\mathfrak K}\ar[rd]_{\mathfrak L}
&\catc[\catw^{-1}]
\ar@{-->}[d]^{\exists !\mathfrak L'}
\\
&\catd.}$$
\end{enumerate}
\end{defn}

\begin{thm}\label{thm170501a}
The category $A$ is (naturally isomorphic to) the 
category of fractions 
$\catf(A)[\catw(A)^{-1}]$,
where the
functor $\mathfrak A_A\colon\catf(A)\to A$ from Definition~\ref{defn170501b} plays the role of the
universal functor $\mathfrak K$ from Definition~\ref{defn170501a}.
\end{thm}

\begin{proof}
First, let $\hat f\colon (x_n)\to(y_m)$ be a morphism in $\catw(A)$; we show that $\mathfrak A_A(\hat f)$ is an isomorphism in $A$.
Lemma~\ref{lem190110a} says that there is an invertible $u\in A$ such that 
$\prod_my_m=u\prod_nx_n$. The fact that $u$ is invertible implies that $\prod_my_m\mid\prod_nx_n$ and $\prod_nx_n\mid\prod_my_m$
so the morphism $\prod_nx_n\mid\prod_my_m$ in $A$ is an isomorphism. That is,
the morphism $\mathfrak A_A(\hat f)$ is invertible, as desired.

Next, let 
$\catd$ be a category, and let $\mathfrak L\colon 
\catf(A)
\to\catd$ be a functor
such that for each $\hat f\in\catw(A)$ the morphism $\mathfrak L(\hat f)$ is an isomorphism in $\catd$.
For each monoid element $x\in A$, define $\mathfrak L'(x)=\mathfrak L((x))$ where $(x)$ is the associated 1-tuple in $\catf(A)$. 
In other words, we define $\mathfrak L'=\mathfrak L\mathfrak B_A$ where $\mathfrak B_A\colon A\to\catf(A)$ 
is the functor from Definition~\ref{defn170501c}.
This yields the next commutative diagram
\begin{equation}\label{eq170501b}
\begin{split}
\xymatrix{
\catf(A)\ar[r]^-{\mathfrak A_A}\ar[rd]_{\mathfrak L}
&A\ar@{-->}[d]^{\mathfrak L\mathfrak B_A=\mathfrak L'}
\\
&\catd}
\end{split}
\end{equation}
as follows. 
First, we need to show that 
for each object $(x_n)$ in $\catf(A)$, we have $\mathfrak L(\mathfrak B_A(\mathfrak A_A((x_n))))\cong\mathfrak L((x_n))$.
By the definitions of $\mathfrak B_A$ and $\mathfrak A_A$, this is tantamount to showing that $\mathfrak L((\prod_nx_n))\cong\mathfrak L((x_n))$.
Note that by 
Remark~\ref{disc170501a}, the factorization
morphism $\hat f\colon (\prod_nx_n)\to(x_n)$ in $\catf(A)$ is in $\catw(A)$. 
Thus, by assumption the 
induced morphism $\mathfrak L(\hat f)\colon \mathfrak L((\prod_nx_n))\to\mathfrak L((x_n))$ is an isomorphism in $\catd$, as desired.

Next, we need to show that for each 
morphism $\hat f\colon(x_n)\to(y_m)$ in $\catf(A)$, we have $\mathfrak L(\mathfrak B_A(\mathfrak A_A(\hat f)))\cong\mathfrak L(\hat f)$.
As in the previous paragraph, the natural morphisms $(\prod_nx_n)\to(x_n)$ and $(\prod_my_m)\to(y_m)$ in $\catf(A)$ are in $\catw(A)$.
Also, the divisibility 
relation $x_n\mid\prod_{m\in f^{-1}(n)}y_m$ implies
that $\prod_nx_n\mid \prod_my_m$.
By Corollary~\ref{cor190115a}, 
this yields a commutative diagram in $\catf(A)$ with 1-tuples on the left
$$\xymatrix{
(\prod_nx_n)\ar[r]\ar[d]_{\mathfrak B_A(\mathfrak A_A(\hat f))}
&(x_n)\ar[d]^{\hat f}\\
(\prod_my_m)\ar[r]
&(y_m)
}$$
where the horizontal morphisms are in $\catw(A)$.
An application of $\mathfrak L$ turns the horizontal morphisms into isomorphisms in $\catd$, by assumption, so we have the next commutative diagram in $\catd$.
$$\xymatrix{
\mathfrak L((\prod_nx_n))\ar[r]^-\cong\ar[d]_{\mathfrak L(\mathfrak B_A(\mathfrak A_A(\hat f)))=\mathfrak L'(\mathfrak A_A(\hat f))}
&\mathfrak L((x_n))\ar[d]^{\mathfrak L(\hat f)}\\
\mathfrak L((\prod_my_m))\ar[r]^-\cong
&\mathfrak L((y_m))
}$$
This is exactly what it means to have $\mathfrak L(\mathfrak B_A(\mathfrak A_A(\hat f)))\cong\mathfrak L(\hat f)$. 
This shows that $\mathfrak L'$ is a functor making the diagram in~\eqref{eq170501b} commute.

Lastly, we need to show that $\mathfrak L'$ is the unique functor (up to natural 
isomorphism, of course) making diagram~\eqref{eq170501b} commute.
Suppose that $\mathfrak L''$ is another functor making the next diagram commute.
\begin{equation}\label{eq170501c}
\begin{split}
\xymatrix{
\catf(A)\ar[r]^-{\mathfrak A_A}\ar[rd]_{\mathfrak L}
&A\ar[d]^-{\mathfrak L''}
\\
&\catd}
\end{split}
\end{equation}
In the next display, the first equality is from Proposition~\ref{prop170430b}.
$$\mathfrak L''=\mathfrak L''\mathfrak{A}_A\mathfrak{B}_A=\mathfrak L\mathfrak{B}_A=\mathfrak L'
$$
The second equality is from diagram~\eqref{eq170501c}, and the third equality is by the definition of $\mathfrak L'$.
This establishes the desired uniqueness.
\end{proof}

Next, we exhibit further properties of the set $\catw(A)$,
beginning with some terminology from~\cite[Chap.~2]{MR2102294}.

\begin{defn}\label{defn171015a}
Let $\textbf{C}$ be a category. 
A collection of morphisms $\catw$ from $\textbf{C}$ is a collection of \textit{weak equivalences} 
if (1) every isomorphism is in $\catw$, and 
(2) $\catw$ satisfies the following $2$-of-$3$ property: if any two of the morphisms $f, g, g \circ f$ are in $\mathcal W$, then the third is also.
\end{defn}

\begin{prop} \label{wk_equiv}
The set $\catw(A)$ is a collection of weak equivalences. 
\end{prop}

\begin{proof}
The isomorphisms of $\catf(A)$ are all in $\mathcal W(A)$ by Remark~\ref{disc170501a}.
For the 2-of-3 condition, let $\hat{f}\colon(x_n) \to (y_m)$ and $\hat{g}\colon(y_m) \to (z_\ell)$ be morphisms in $\catf(A)$.
Let $r,s\in A$ be such that $\prod_my_m=r\prod_nx_n$ and $\prod_\ell z_\ell=s\prod_my_m=rs\prod_nx_n$.
The product $rs$ is invertible in $A$ if and only if $r$ and $s$ are both invertible in $A$.
So, using the cancellativity of $A$, one applies Lemma~\ref{lem190110a}
to conclude that $\hat g\circ\hat f\in\catw(A)$ if and only if both $\hat f$ and $\hat g$ are in $\catw(A)$.
This implies the 2-of-3 property.
\end{proof}

Next, we turn our attention to the calculus of fractions of~\cite{MR0210125}.

\begin{defn}\label{defn190106a}
Let $\textbf{C}$ be a category. 
A collection of morphisms $\catw$ from $\textbf{C}$  \textit{admits a calculus of right fractions}
if it satisfies the following conditions:
\begin{enumerate}[(1)]
\item \label{defn190106a1} 
$\catw$ contains all identities and is closed under composition;
\item \label{defn190106a2}
(right Ore condition) for each morphism $f\colon x\to y$ in $\catw$ and every morphism $g\colon z\to y$ in $\textbf{C}$, 
there are morphisms $f'\colon w\to z$ in $\catw$ and $g'\colon w\to x$ in $\textbf{C}$ 
$$\xymatrix{
w \ar@{-->}[r]^-{g'}\ar@{-->}[d]_{f'}
&x\ar[d]^{f}\\
z\ar[r]^-{g}&y}$$
making the above diagram commute; and
\item \label{defn190106a3}
(right cancellability) for every pair of morphisms $f,f'\colon x\to y$ in $\textbf{C}$,
if there is a morphism $g\colon y\to z$ in $\catw$ such that $g\circ f=g\circ f'$, then there is a morphism
$h\colon w\to x$ in $\catw$ such that $f\circ h=f'\circ h$. 
\end{enumerate}
\end{defn}

\begin{prop} \label{prop190106a}
The set $\catw(A)$ admits a calculus of right fractions.
\end{prop}

\begin{proof}
The set $\catw(A)$ contains all identities and is closed under composition by Proposition~\ref{wk_equiv}.
For the right Ore condition, let $\hat f\colon(x_n)\to(y_m)$ be a morphism in $\catw(A)$, and let
$\hat g\colon(z_p)\to(y_m)$ be a morphism in $\catf(A)$. 
The morphism condition for $\hat g$ implies that $\prod_pz_p\mid\prod_my_m$.
And according to Lemma~\ref{lem190110a}
the fact that $\hat f$ is in $\catw(A)$ implies that $\prod_my_m=r\prod_nx_n$ for an invertible element $r\in A$.
It follows that $\prod_pz_p\mid\prod_nx_n$, so the unique map $g'\colon[N]\to[1]$ defines a morphism $\widehat{g'}\colon(\prod_pz_p)\to(x_n)$.
Also, the unique map $f'\colon [P]\to[1]$ defines a factorization morphism $\widehat{f'}\colon(\prod_pz_p)\to(z_p)$ in $\catw(A)$
by Remark~\ref{disc170501a}.
Corollary~\ref{cor190115a} shows that these morphisms
make the following diagram commute
$$\xymatrix{
(\prod_pz_p) \ar@{-->}[r]^-{\widehat{g'}}\ar@{-->}[d]_{\widehat{f'}}
&(x_n)\ar[d]^{\hat f}\\
(z_p)\ar[r]^-{\hat g}&(y_m)}$$
so the right Ore condition is satisfied.

For right cancellability, let  $\hat f,\widehat{f'}\colon (x_n)\to (y_m)$ be morphisms in $\catf(A)$, and let
$\hat g\colon (y_m)\to (z_p)$ be a morphism in $\catw(A)$ such that $\hat g\circ \hat f=\hat g\circ \widehat{f'}$.
As in the preceding paragraph, one shows that the unique morphism
$\hat h\colon (\prod_nx_n)\to (x_n)$, which is in $\catw(A)$, satisfies $\hat f\circ \hat h=\widehat{f'}\circ\hat h$, as desired.
\end{proof}

\begin{disc}\label{disc190106a}
Because of~\cite[Sections I.2--3]{MR0210125}
one of the benefits of $\catw(A)$ admitting a calculus of right fractions is that it allows for a simple representation of morphisms
in $\catf(A)[\catw(A)^{-1}]$ as simple fractions instead of more complicated zig-zag diagrams. Furthermore, it provides a 
straightforward characterization of equality of parallel morphisms; and the natural functor $\catf(A)\to\catf(A)[\catw(A)^{-1}]$ is flat,
moreover, it is uniformly flat~\cite[Section~6]{MR1406095}.
In particular, the properties described in~\cite[Section~1]{MR1406095} are satisfied, including that $\catw(A)$ is focal.
Since we already have a simple description of $\catf(A)[\catw(A)^{-1}]$ in Theorem~\ref{thm170501a}, we resist the urge to say
more about these conditions here.
\end{disc}

\section{Weak Divisibility}\label{sec170708a}

\begin{assumptions}
In this section, again  $A$ is a divisibility monoid.
In addition, we say that a morphism $\hat f$ is a \emph{weak equivalence} provided that $\hat f\in\catw(A)$.
\end{assumptions}

Here
we introduce a notion of 
weak divisibility
of morphisms for use with our treatment of 
weak primeness
in
Section~\ref{sec170702a}; see Definition~\ref{defn170702a}.
It may not be clear initially why this definition is useful; see however the application of 
weak
primeness in Theorem~\ref{prop170710a}.
For some motivation/perspective on this definition, 
we begin with a motivating example.

\begin{ex}\label{ex170706azzz}
Consider the divisibility monoid $A=\bbz\setminus\{0\}$
and the morphisms 
$\widehat{\id_{[1]}}^{(2)}_{(6)}\colon(2)\to(6)$ and $\widehat{\id_{[1]}}^{(5)}_{(105)}\colon(5)\to(105)$.
Intuitively, we think of the morphism $\hat f$ as 
multiplication by 3, while $\hat g$ is multiplication by 
$21$.
With this guiding 
thought,
it seems reasonable that $\hat f$ 
should divide
$\hat g$ and that $\hat g$ should not divide $\hat f$.

A general idea says that 
$\hat f$ should divide $\hat g$ when
$\hat g$ is in 
a nice subset  $\langle\hat f\rangle\subseteq\operatorname{Mor}(\catf(A))$ that behaves like an
ideal generated by $\hat f$;
the 
difficulty
is to decide what $\langle\hat f\rangle$ should be.
In our opinion, it should contain $\hat f$, and it should be closed under meaningful compositions.
So, for instance, one could define  
$\langle\hat f\rangle$ to be the set of all meaningful compositions 
$\hat\alpha\circ\hat f\circ\hat\beta$.
With this first approximation, in order 
for $\hat f$ to divide $\hat g$
in our example, we would require the existence
of morphisms in $\catf(A)$ making the following diagram commute.
\begin{equation}
\label{eq170918a}
\begin{split}
\xymatrix{
(5)\ar@{-->}[r]\ar[d]_{\hat g}
&(2)\ar[d]^-{\hat f}
\\
(5\cdot 7\cdot 3)
&(2\cdot 3)\ar@{-->}[l]
}
\end{split}
\end{equation}
It is straightforward to show that no such morphisms exist, e.g., because $5\nmid 2$.
In other words, the failure of existence of such a diagram lies in the
entries of the domains and codomains of the morphisms $\hat f$ and $\hat g$, not necessarily in the 
factorization properties encoded in the morphisms.

One can
come close to remedying this by tensoring
(i.e., concatenating, as in Definition~\ref{defn170605a})
appropriately to make the domains line up
as in the following  diagram wherein the top horizontal arrow is  the natural isomorphism.
$$\xymatrix@C=1em{
(2,5)\ar@{=}[r]&(2)\otimes(5)\ar[rr]\ar[d]_{(2)\otimes\hat g}
&&(5)\otimes(2)\ar[d]^-{(5)\otimes\hat f}\ar@{=}[r]
&(5,2)
\\
(2,105)\ar@{=}[r]&(2)\otimes(5\cdot 7\cdot 3)
&&(5)\otimes(2\cdot 3)\ar@{-->}[ll]\ar@{=}[r]
&(5,6).
}
$$
However, one again checks readily that there is no morphism for the dashed arrow, let alone one making the 
diagrams
commute.

Taking a cue from 
the category of fractions,
one could return to diagram~\eqref{eq170918a} and replace the horizontal
arrows with zig-zags as in the following
$$\xymatrix{
(5)\ar[d]_-{\hat g}
&\ar[l]_-{\widehat{f_1}}^-{\in\catw(A)}(c_{n_1})\ar[r]^-{\widehat{f_2}}
&(c_{n_2})
&\ar[l]_-{\widehat{f_3}}^-{\in\catw(A)}\cdots\ar[r]^-{\widehat{f_{2r}}}
&(2)\ar[d]^-{\hat f}
\\
(5\cdot 7\cdot 3)
&\ar[l]_-{\widehat{g_{2s}}}\cdots\ar[r]^-{\widehat{g_{3}}}_-{\in\catw(A)}
&(d_{m_{2}})
&\ar[l]_-{\widehat{g_{2}}}(d_{m_1})\ar[r]^-{\widehat{g_{1}}}_-{\in\catw(A)}
&(2\cdot 3)
}$$
where each row consists of a finite  number of morphisms, and the 
odd-indexed
morphisms $\widehat{f_{2k+1}}$ and $\widehat{g_{2k+1}}$ 
are weak equivalences.
(Note that we make no assumption about commutativity in this diagram, because the shape makes commutativity meaningless.)
However, there is no such diagram.
Indeed,
were such a diagram to exist, 
Remark~\ref{disc170702a} would provide the following \emph{commutative} diagram in $A$
$$\xymatrix{
5\ar@{-->}[r]\ar[d]
&2\ar[d]
\\
5\cdot 7\cdot 3
&2\cdot 3\ar@{-->}[l]
}$$
but there is no such diagram (let alone a commutative one),
e.g., because $5\nmid 2$.
\end{ex}

In light of Example~\ref{ex170706azzz}, we introduce the following
which is a fusion of the two inadequate ways we suggested that one might 
address the deficiencies of diagram~\eqref{eq170918a}.

\begin{defn}\label{defn170702a}
Recall the concatenation tensor product from Definition~\ref{defn170605a}.
Let $\hat f\colon (v_i)\to(w_j)$ and $\hat g\colon (x_n)\to(y_m)$ be morphisms in $\catf(A)$.
We say that $\hat f$ 
\emph{weakly divides} $\hat g$, denoted $\hat f\divs\hat g$,
provided that there are tuples $(a_p)$ and $(b_q)$ in $\catf(A)$
and a diagram
in $\catf(A)$
\begin{equation}
\label{diag170702a}
\begin{split}
\xymatrix{
(a_p)\otimes(x_n)\ar[d]_-{(a_p)\otimes\hat g}
&\ar[l]_-{\widehat{f_1}}^-{\in\catw(A)}(c_{n_1})\ar[r]^-{\widehat{f_2}}
&(c_{n_2})
&\ar[l]_-{\widehat{f_3}}^-{\in\catw(A)}\cdots\ar[r]^-{\widehat{f_{2r}}}
&(b_q)\otimes(v_i)\ar[d]^-{(b_q)\otimes\hat f}
\\
(a_p)\otimes(y_m)
&\ar[l]_-{\widehat{g_{2s}}}\cdots\ar[r]^-{\widehat{g_{3}}}_-{\in\catw(A)}
&(d_{m_{2}})
&\ar[l]_-{\widehat{g_{2}}}(d_{m_1})\ar[r]^-{\widehat{g_{1}}}_-{\in\catw(A)}
&(b_q)\otimes(w_j)
}
\end{split}
\end{equation}
such that each row consists of a finite  number of morphisms, and the 
odd-indexed
morphisms $\widehat{f_{2k+1}}$ and $\widehat{g_{2k+1}}$ 
are weak equivalences.
\end{defn}

As in Example~\ref{ex170706azzz},
note that we do not say that the diagram~\eqref{diag170702a}
is commutative in $\catf(A)$, since the shape of the diagram makes commutativity meaningless.
Some readers will notice that 
Proposition~\ref{prop190106a} can be used to simplify the diagram~\eqref{diag170702a}.
We do not make this explicit here because of the  simplification in Theorem~\ref{prop170708a} below.

\begin{ex}\label{ex170706a}
Consider the divisibility monoid $A=\bbz\setminus\{0\}$
and the morphisms 
$\widehat{\id_{[1]}}^{(2)}_{(6)}\colon(2)\to(6)$ and $\widehat{\id_{[1]}}^{(5)}_{(105)}\colon(5)\to(105)$
from Example~\ref{ex170706azzz}.
We claim that the following diagram shows that $\widehat{\id_{[1]}}^{(2)}_{(6)}\divs\widehat{\id_{[1]}}^{(5)}_{(105)}$.
$$\xymatrix{
(2)\otimes(5)\ar[d]_-{(2)\otimes\widehat{\id_{[1]}}^{(5)}_{(105)}}
&\ar[l]_-{\widehat{f_1}}^-{\in\catw(A)}(10)\ar[r]^-{\widehat{f_2}}
&(5)\otimes(2)\ar[d]^-{(5)\otimes\widehat{\id_{[1]}}^{(2)}_{(6)}}
\\
(2)\otimes(105)
&\ar[l]_-{\widehat{g_{2}}}(30)\ar[r]^-{\widehat{g_{1}}}_-{\in\catw(A)}
&(5)\otimes(6)
}
$$
Here the horizontal morphisms are all factorization morphisms; see Example~\ref{basic-example} 
and
Definition~\ref{defn170715b}.
Indeed, by definition, this diagram has the following form
\begin{equation}
\label{diag170706b}
\begin{split}
\xymatrix{
(2,5)\ar[d]_-{\widehat{g'}}
&\ar[l]_-{\widehat{f_1}}^-{\in\catw(A)}(10)\ar[r]^-{\widehat{f_2}}
&(5,2)\ar[d]^-{\widehat{f'}}
\\
(2,105)
&\ar[l]_-{\widehat{g_{2}}}(30)\ar[r]^-{\widehat{g_{1}}}_-{\in\catw(A)}
&(5,6)
}
\end{split}
\end{equation}
where 
the factorization morphisms $\widehat{f_1}$ and $\widehat{g_1}$ are weak equivalences by 
Remark~\ref{disc170501a}.
\end{ex}

\begin{disc}\label{disc170702a}
Let $\hat f\colon (v_i)\to(w_j)$ and $\hat g\colon (x_n)\to(y_m)$ be morphisms in $\catf(A)$
such that $\hat f\divs\hat g$.
The diagram~\eqref{diag170702a} provided by
Definition~\ref{defn170702a}
induces the following diagram in $A$
$$\xymatrix{
\prod_pa_p\cdot\prod_nx_n\ar[d]^-{\mathfrak A_A((a_p)\otimes\hat g)}\ar[r]^-{\mathfrak A_A(\widehat{f_1})^{-1}}
&\prod_{n_1}c_{n_1}\ar[r]^-{\mathfrak A_A(\widehat{f_2})}
&\prod_{n_2}c_{n_2}\ar[r]^-{\mathfrak A_A(\widehat{f_3})^{-1}}
&\cdots\ar[r]^-{\mathfrak A_A(\widehat{f_{2r}})}
&\prod_qb_q\cdot\prod_iv_i\ar[d]_-{\mathfrak A_A((b_q)\otimes\hat f)}
\\
\prod_pa_p\cdot\prod_my_m
&\cdots\ar[l]_-{\mathfrak A_A(\widehat{g_2})}
&\prod_{m_2}d_{m_2}\ar[l]_-{\mathfrak A_A(\widehat{g_3})^{-1}}
&\prod_{m_1}d_{m_1}\ar[l]_-{\mathfrak A_A(\widehat{g_{2}})}
&\prod_qb_q\cdot\prod_jw_j\ar[l]_-{\mathfrak A_A(\widehat{g_1})^{-1}}
}$$
where $\mathfrak A_A\colon\catf(A)\to A$ is the the functor from Definition~\ref{defn170501b}.
Note that 
the morphisms $\mathfrak A_A(\widehat{f_{2k+1}})$ and $\mathfrak A_A(\widehat{g_{2k+1}})$ 
in this diagram are isomorphisms in $\catf(A)$ since  the morphisms $\widehat{f_{2k+1}}$ and $\widehat{g_{2k+1}}$ 
in the original diagram~\eqref{diag170702a}  are weak equivalences; see Theorem~\ref{thm170501a}.
In particular, in $A$ we have  $\prod_pa_p\cdot\prod_nx_n\mid\prod_qb_q\cdot\prod_iv_i$ and
$\prod_qb_q\cdot\prod_jw_j\mid\prod_pa_p\cdot\prod_my_m$.

Note furthermore that this new diagram in $A$ commutes by 
Remark~\ref{disc171015a}.

Lastly, we observe that one can detect weak divisibility using diagrams of shapes different
from~\eqref{diag170702a} by inserting identity morphisms in the horizontal sequences.
For instance, if one has a diagram of the following shape in $\catf(A)$
$$\xymatrix{
(a_p)\otimes(x_n)\ar[d]_-{(a_p)\otimes\hat g}
&\ar[l]_-{\widehat{f'}}^-{\in\catw(A)}(c_{n_1})
&\ar[l]_-{\widehat{f''}}^-{\in\catw(A)}(b_q)\otimes(v_i)\ar[d]^-{(b_q)\otimes\hat f}
\\
(a_p)\otimes(y_m)
&\ar[l]_-{\widehat{g_{2}}}(d_{m_{1}})\ar[r]^-{\widehat{g_{1}}}_-{\in\catw(A)}
&(b_q)\otimes(w_j)
}$$
then one obtains the next 
diagrams
to conclude that $\hat f\divs\hat g$.
\begin{gather*}
\xymatrix@C=11mm{
(a_p)\otimes(x_n)\ar[d]_-{(a_p)\otimes\hat g}
&\ar[l]_-{\widehat{f'}}^-{\in\catw(A)}(c_{n_1})
&\ar[l]_-{\widehat{f''}}^-{\in\catw(A)}(b_q)\otimes(v_i)\ar[r]^-{\id_{(b_q)\otimes(v_i)}}\ar@/_2pc/[ll]_-{\widehat{f'}\circ\widehat{f''}}^-{\in\catw(A)}
&(b_q)\otimes(v_i)\ar[d]^-{(b_q)\otimes\hat f}
\\
(a_p)\otimes(y_m)
&\ar[l]_-{\widehat{g_{2}}}(d_{m_{1}})\ar[rr]^-{\widehat{g_{1}}}_-{\in\catw(A)}
&&(b_q)\otimes(w_j)}\\
\xymatrix@C=11mm{
(a_p)\otimes(x_n)\ar[d]_-{(a_p)\otimes\hat g}
&\ar[l]_-{\widehat{f'}\circ\widehat{f''}}^-{\in\catw(A)}(b_q)\otimes(v_i)\ar[r]^-{\id_{(b_q)\otimes(v_i)}}
&(b_q)\otimes(v_i)\ar[d]^-{(b_q)\otimes\hat f}
\\
(a_p)\otimes(y_m)
&\ar[l]_-{\widehat{g_{2}}}(d_{m_{1}})\ar[r]^-{\widehat{g_{1}}}_-{\in\catw(A)}
&(b_q)\otimes(w_j)
}
\end{gather*}
\end{disc}

In the next result, 
condition~\eqref{prop170708a5}  makes detection of weak divisibility quite easy.
And it gives another indication of how the morphisms in $\catw(A)$ can see factorization properties in $A$.
Compare condition~\eqref{prop170708a2'}  to 
diagram~\eqref{diag170706b}
in 
Example~\ref{ex170706a}.

\begin{thm}\label{prop170708a}
Let
$\hat f\colon (v_i)\to(w_j)$ and $\hat g\colon (x_n)\to(y_m)$ 
be morphisms in $\catf(A)$. Then the following conditions are equivalent.
\begin{enumerate}[\rm(i)]
\item \label{prop170708a1}
We have $\hat f\divs\hat g$.
\item \label{prop170708a2'}
There are elements $a,b\in A$  and a diagram
in $\catf(A)$
\begin{equation*}
\xymatrix{
(a)\otimes(x_n)\ar[d]_-{(a)\otimes\hat g}
&\ar[l]_-{\hat{\mu}}^-{\in\catw(A)}\left(a\prod_nx_n\right)\ar[r]^-{\widehat{\alpha}}
&(b)\otimes(v_i)\ar[d]^-{(b)\otimes\hat f}
\\
(a)\otimes(y_m)
&\ar[l]_-{\widehat{\beta}}\left(b\prod_jw_j\right)\ar[r]^-{\hat{\eta}}_-{\in\catw(A)}
&(b)\otimes(w_j)
}
\end{equation*}
where the middle tuple in each row has length 1, and the morphisms $\hat\mu$ and $\hat\eta$ are 
factorization morphisms.
\item \label{prop170708a4}
We have $\prod_nx_n\cdot\prod_jw_j\mid\prod_iv_i\cdot\prod_my_m$.
\item \label{prop170708a3}
There are elements $a,b\in A$ such that $a\prod_nx_n\mid b\prod_iv_i$ and $b\prod_jw_j\mid a\prod_my_m$.
\item \label{prop170708a5}
The monoid elements $r,s\in A$ such that $\prod_my_m=r\prod_nx_n$ and $\prod_jw_j=s\prod_iv_i$ satisfy $s\mid r$.
\end{enumerate}
\end{thm}

\begin{proof}
The 
implication~$\eqref{prop170708a2'}\implies\eqref{prop170708a1}$
follows by definition of $\hat f\divs\hat g$,
and the implication~$\eqref{prop170708a1}\implies\eqref{prop170708a3}$
is from the first paragraph of Remark~\ref{disc170702a} with $a=\prod_pa_p$ and $b=\prod_qb_q$.

$\eqref{prop170708a3}\implies\eqref{prop170708a4}$
Assume that there are elements $a,b\in A$ such that $a\prod_nx_n\mid b\prod_iv_i$ and $b\prod_jw_j\mid a\prod_my_m$.
Multiply these divisibility relations by $\prod_jw_j$ and $\prod_iv_i$, respectively, to conclude that
$$
\textstyle
a\prod_nx_n\cdot\prod_jw_j\mid b\prod_iv_i\cdot\prod_jw_j
\mid a\prod_iv_i\cdot\prod_my_m.
$$
It follows that $a\prod_nx_n\cdot\prod_jw_j\mid a\prod_iv_i\cdot\prod_my_m$,
hence $\prod_nx_n\cdot\prod_jw_j\mid \prod_iv_i\cdot\prod_my_m$ since $A$ is cancellative.

$\eqref{prop170708a4}\implies\eqref{prop170708a2'}$
Assume that $\prod_nx_n\cdot\prod_jw_j\mid\prod_iv_i\cdot\prod_my_m$.
This yields the  morphism $\widehat{\id_{[1]}}$ in the following diagram in $\catf(A)$.
$$
\xymatrix@C=7mm{
\left(\prod_iv_i\right)\otimes(x_n)\ar[d]^-{(\prod_iv_i)\otimes\hat g}
&&\ar[ll]_-{\hat{\mu}}^-{\in\catw(A)}\left(\prod_iv_i\cdot\prod_nx_n\right)\ar[r]^-{\widehat{\alpha}}
&\left(\prod_nx_n\right)\otimes(v_i)\ar[d]^-{(b)\otimes\hat f}
\\
\left(\prod_iv_i\right)\otimes(y_m)
&\ar[l]_-{\widehat{\gamma}}\left(\prod_iv_i\cdot\prod_my_m\right)
&\ar[l]_-{\widehat{\id_{[1]}}}\left(\prod_nx_n\cdot\prod_jw_j\right)\ar[r]^-{\hat{\eta}}_-{\in\catw(A)}
&\left(\prod_nx_n\right)\otimes(w_j)
}
$$
The morphisms $\hat\mu$, $\hat\alpha$, $\hat\eta$, and $\hat\gamma$ are factorization morphisms.
Thus, condition~\eqref{prop170708a2'} is satisfied with
$a=\prod_iv_i$
and
$b=\prod_nx_n$.

$\eqref{prop170708a4}\iff\eqref{prop170708a5}$
From the conditions $\prod_my_m=r\prod_nx_n$ and $\prod_jw_j=s\prod_iv_i$ it is straightforward to show that 
$\prod_nx_n\cdot\prod_jw_j\mid\prod_iv_i\cdot\prod_my_m$ if and only if $s\mid r$.
\end{proof}

Next, we present
some applications of Theorem~\ref{prop170708a}.
In the first one, parts~\eqref{prop170707a1}--\eqref{prop170707a2} say that
weak divisibility is a pre-order on the set $\operatorname{Mor}(\catf(A))$.
Part~\eqref{prop170707a3} says that weak equivalences are strongly minimal in this ordering;
see Theorem~\ref{prop170708b} for more about this.
It is not hard to prove Proposition~\ref{prop170707a} by definition, without invoking Theorem~\ref{prop170708a}, but 
that proof is longer, and this one showcases Theorem~\ref{prop170708a} nicely.

\begin{prop}\label{prop170707a}
Let
$\hat f\colon (v_i)\to(w_j)$ and $\hat g\colon (x_n)\to(y_m)$ and $\hat h\colon (z_\ell)\to(t_k)$
be morphisms in $\catf(A)$.
\begin{enumerate}[\rm(a)]
\item \label{prop170707a1}
We have $\hat f\divs\hat f$.
\item \label{prop170707a2}
If $\hat f\divs\hat g$ and $\hat g\divs\hat h$, then $\hat f\divs\hat h$.
\item \label{prop170707a3}
If $\hat f$ is a weak equivalence, e.g., if $\hat f$ is an isomorphism, then $\hat f\divs\hat g$. 
\end{enumerate}
\end{prop}

\begin{proof}
Let $r,s,q\in A$ be such that $\prod_my_m=r\prod_nx_n$ and $\prod_jw_j=s\prod_iv_i$ and $\prod_kt_k=q\prod_\ell z_\ell$.
Part~\eqref{prop170707a1}
follows from Theorem~\ref{prop170708a} because $s\mid s$.
Part~\eqref{prop170707a2} follows similarly: if $\hat f\divs\hat g$ and $\hat g\divs\hat h$, then 
$s\mid r$ and $r\mid q$, hence $s\mid q$, so $\hat f\divs\hat h$.

\eqref{prop170707a3}
Assume that $\hat f$ is a weak equivalence.
Lemma~\ref{lem190110a} implies that $s$ is invertible, so $s\mid r$,
hence $\hat f\divs\hat g$ by Theorem~\ref{prop170708a}.
\end{proof}

Our next  application of Theorem~\ref{prop170708a}
complements Proposition~\ref{prop170707a}\eqref{prop170707a3}.

\begin{thm}\label{prop170708b}
Let
$\hat f\colon (v_i)\to(w_j)$ 
be a morphism  in $\catf(A)$.
The following conditions are equivalent.
\begin{enumerate}[\rm(i)]
\item \label{prop170708b2a}
We have 
$\hat f\divs\hat g$ for all morphisms $\hat g$ in $\catf(A)$.
\item \label{prop170708b2b}
We have 
$\hat f\divs\hat g$ for some weak equivalence $\hat g$ in $\catf(A)$.
\item \label{prop170708b2c}
The morphism $\hat f$ is a weak equivalence.
\end{enumerate}
\end{thm}

\begin{proof}
The  implication~$\eqref{prop170708b2a}\implies\eqref{prop170708b2b}$ is trivial,
and~$\eqref{prop170708b2c}\implies\eqref{prop170708b2a}$ is from Proposition~\ref{prop170707a}\eqref{prop170707a3}.

\eqref{prop170708b2b}$\implies$\eqref{prop170708b2c} Assume that there is a weak equivalence $\hat g\colon (x_n)\to(y_m)$
in $\catf(A)$  such that $\hat f\divs\hat g$.
Let $r,s\in A$ be such that $\prod_my_m=r\prod_nx_n$ and $\prod_jw_j=s\prod_iv_i$.
Lemma~\ref{lem190110a} implies that $r$ is invertible.
Theorem~\ref{prop170708a} says that $s$ divides the invertible element $r$, so $s$ is invertible as well,
so another application of Lemma~\ref{lem190110a} implies that $\hat f$ is a weak equivalence.
\end{proof}

Proposition~\ref{prop170707a}
shows that the weak divisibility relation on the set of morphisms in
$\catf(A)$ is a pre-order. Our next definition is the induced equivalence relation.

\begin{defn}\label{defn170715a}
Let $\hat f\colon (v_i)\to(w_j)$ and $\hat g\colon (x_n)\to(y_m)$ be morphisms in $\catf(A)$.
We say that $\hat f$ and $\hat g$ are
\emph{weakly associate}
if
$\hat f\divs\hat g$ and $\hat g\divs\hat f$.
\end{defn}

The notions from this section provide several characterizations of the weakly associate relation.
We single out the next 
ones,
which follows directly from Theorem~\ref{prop170708a}, for 
perspective
in Theorem~\ref{prop170710y} below.

\begin{prop}\label{prop170715a}
Let
$\hat f\colon (v_i)\to(w_j)$ and $\hat g\colon (x_n)\to(y_m)$ 
be morphisms in $\catf(A)$. 
Let $r,s\in A$ be such that $\prod_my_m=r\prod_nx_n$ and $\prod_jw_j=s\prod_iv_i$.
Then the following conditions are equivalent.
\begin{enumerate}[\rm(i)]
\item \label{prop170715a1}
The morphisms $\hat f$ and $\hat g$ are weakly associate.
\item \label{prop170715a4}
The elements $r$ and $s$ divide each other in $A$, i.e., we have $r\mid s$ and $s\mid r$.
\item \label{prop170715a5}
We have $r=vs$ for some invertible element $v\in A$.
\item \label{prop170715a2}
The elements $\prod_nx_n\cdot\prod_jw_j$ and $\prod_iv_i\cdot\prod_my_m$ divide each other in $A$.
\item \label{prop170715a3}
We have $\prod_nx_n\cdot\prod_jw_j=u\cdot\prod_iv_i\cdot\prod_my_m$
for some invertible element $u\in A$.
\end{enumerate}
\end{prop}

\section{Weak
Irreducibility}\label{sec170501a}

\begin{assumptions}
In this section, again  $A$ is a divisibility monoid,
and we say that a morphism $\hat f$ is a weak equivalence provided that $\hat f\in\catw(A)$.
\end{assumptions}

Here
we treat a notion of irreducibility  for morphisms and objects in $\catf(A)$.
Note that the definition of an irreducible tuple 
fits with a fundamental idea behind our construction, that factorization in $A$ is tracked by morphisms in $\catf(A)$.

\begin{defn}\label{defn170501e}
A monoid element $a\in A$ is \emph{irreducible} if it is not invertible and only has trivial factorizations in $A$: if $a=bc$ in $A$, then either $b$
or $c$ is invertible in $A$.
We shall say that a morphism $\hat{f}$ in $\mathcal{F}(A)$ is \textit{weakly irreducible} when $\hat{f}$ is not a weak equivalence and for every factorization 
$\hat{f} = \hat{g} \circ \hat{h}$ either $\hat{g}$ or $\hat{h}$ is a weak equivalence. 
We say an object $(x_n) \in \mathcal{F}(A)$ is 
\textit{weakly irreducible} 
if  the
morphism $(1) \to (x_n)$ is 
weakly
irreducible.
\end{defn}

\begin{disc}
\label{disc170501c}
If $\hat f\colon(x_n)\to(y_m)$ is a morphism in $\catf(A)$ with $(x_n)=\emptytuple$ or $(y_m)=\emptytuple$,
then  $\hat f$ is a weak equivalence, so in particular $\hat f$ is not 
weakly
irreducible. 
For example, this shows that $\emptytuple$ is not 
weakly irreducible  in $\catf(A)$.

The term ``weakly irreducible morphism'' is meant to suggest the usual notion of irreducibility,
modified by weak equivalences.
Of course, one could introduce variations on this,
for instance, a stronger notion declaring a morphism $\hat{f}$ in $\mathcal{F}(A)$ to be ``irreducible'' when $\hat{f}$ is not an isomorphism and for every factorization 
$\hat{f} = \hat{g} \circ \hat{h}$ either $\hat{g}$ or $\hat{h}$ is an isomorphism. 
However, this version is too strong for our purposes.
\end{disc}

\begin{lem}\label{lem190115a}
Consider morphisms in $\catf(A)$
$$(v_i)\xra[\in\catw(A)]{\hat\epsilon}(r_a)\xra{\hat{\eta}}(s_b)
\xra[\in\catw(A)]{\hat\phi}(w_j)$$ 
such that $\hat\epsilon$ and $\hat\phi$ are weak equivalences.
If the composition $\hat f=\hat\phi\circ\hat\eta\circ\hat\epsilon$
is weakly irreducible, then so is $\hat\eta$.
\end{lem}

\begin{proof}
Assume that $\hat f$ is weakly irreducible.
Then $\hat f$ is not a weak equivalence, so the 2-of-3 condition for $\catw(A)$ implies that $\hat\eta$
is not a weak equivalence.
Suppose by way of contradiction that $\hat\eta$ is not weakly irreducible, so that it factors as
$\hat\eta=\hat h\circ\hat g$  where neither $\hat h$ nor $\hat g$ is a weak equivalence. 
Plug this factorization into the definition of $\hat f$ to conclude that
$$\hat f
=\hat\phi\circ\hat\eta
\circ\hat\epsilon
=(\hat\phi\circ
\hat h)\circ(\hat g\circ\hat\epsilon)
$$
Since $\hat\epsilon$ and $\hat\phi$ are both weak equivalences, 
the 2-of-3 condition for $\catw(A)$ implies that neither $\hat\phi\circ
\hat h$ nor $\hat g\circ\hat\epsilon$ is a weak equivalence, contradicting the fact that $\hat f$ is weakly irreducible.
This establishes the result.
\end{proof}

The next result shows, e.g., that the morphism $\widehat{\id_{[1]}}\colon(2)\to(6)$ from
Example~\ref{ex170706azzz}
is weakly irreducible.
And it gives another indication of how the morphisms in $\catw(A)$ can see factorization properties in $A$.

\begin{thm}\label{lem170501a} 
For an element $r \in A$, the following conditions are equivalent.
\begin{enumerate}[\rm(i)]
\item \label{lem170501a1}
$r$
is irreducible in $A$,
\item \label{lem170501a5}
each morphism $\hat f\colon (x_n)\to (y_m)$ where $r\cdot\prod_nx_n=\prod_my_m$ is weakly irreducible,
\item \label{lem170501a2} the 1-tuple
$(r)$
is 
weakly
irreducible in $\mathcal{F}(A)$, 
and
\item \label{lem170501a0}
some morphism $\hat f\colon (x_n)\to (y_m)$ with $r\cdot\prod_nx_n=\prod_my_m$ is weakly irreducible.
\end{enumerate}
\end{thm}

\begin{proof}
The implications $\eqref{lem170501a5}\implies\eqref{lem170501a2}
\implies\eqref{lem170501a0}$
are straightforward.

$\eqref{lem170501a1}\implies\eqref{lem170501a5}$
Assume that $r$ is irreducible in $A$.
Consider a morphism 
$\hat f\colon (x_n)\to (y_m)$ such that $r\cdot\prod_nx_n=\prod_my_m$.
The fact that $r$ is not invertible implies that the given morphism is not a weak equivalence by Lemma~\ref{lem190110a}.
To show that $\hat f$ is weakly irreducible, decompose $\hat f$ as a composition
$$(x_n)\xra{\hat g}(z_\ell)\xra{\hat h}(y_m).$$
Let $a,b\in A$ be such that $\prod_\ell z_\ell=a\prod_nx_n$ and 
$$\textstyle r\cdot\prod_nx_n=\prod_my_m=b\prod_\ell z_\ell=ab\prod_nx_n.$$
Cancellation implies that $r=ab$.
Since $r$ is irreducible, we conclude that $a$ or $b$ is invertible in $A$,
so $\hat g$ or $\hat h$ is a weak equivalence by Lemma~\ref{lem190110a}, as desired.

$\eqref{lem170501a0}\implies\eqref{lem170501a1}$
Assume that $\hat f\colon (x_n)\to (y_m)$ is a weakly irreducible morphism in $\catf(A)$ with $r\cdot\prod_nx_n=\prod_my_m$.
In particular, $\hat f$ is not a weak equivalence, so $(y_m)\neq\emptytuple$, and hence $(x_n)\neq\emptytuple$.
Consider the decomposition $\hat f=\hat\phi\circ\widehat{\id_{[P]}}^{(z_p)}_{(a_pz_p)}
\circ\hat\epsilon$ from Proposition~\ref{prop170423a}
\begin{equation}\label{eq190114a}
(x_n)\xra[\in\catw(A)]{\hat\epsilon}(z_p)\xra{\widehat{\id_{[P]}}^{(z_p)}_{(a_pz_p)}}
(a_pz_p)
\xra[\in\catw(A)]{\hat\phi}(y_m).
\end{equation}
Recall the superscript/subscript notation for $\widehat{\id_{[P]}}$ from Definition~\ref{defn180217a}.

Claim 1: There is an element $p_0\in [P]$ such that $a_p$ is invertible in $A$ for all $p\neq p_0$.
Indeed, Proposition~\ref{prop170503c} implies that some $a_p$ is non-invertible in $A$.
Let $p_0$ be the smallest such $p\in[P]$, so that $a_{p_0}$ is non-invertible in $A$ but $a_p$ is invertible in $A$ for all $p<p_0$.
The morphism $\widehat{\id_{[P]}}^{(z_p)}_{(a_pz_p)}$ factors as 
the composition of the next two morphisms.
\begin{equation}\label{eq190114b}
(z_p)\xra[\notin\catw(A)]{\widehat{\id_{[P]}}^{(z_p)}_{(z_p')}}
\underbrace{(z_1,\ldots,z_{p_0-1},a_{p_0}z_{p_0},z_{p_0+1},\ldots,z_P)}_{=(z_p')}
\xra{\widehat{\id_{[P]}}^{(z_p')}_{(a_pz_p)}}(a_pz_p)
\end{equation}
The fact that $a_{p_0}$ is non-invertible implies that the first factor $\widehat{\id_{[P]}}^{(z_p)}_{(z_p')}$ is not a weak equivalence. 
Since $\hat\epsilon$ is a weak equivalence, the 2-of-3 condition implies that $\widehat{\id_{[P]}}^{(z_p)}_{(z_p')}\circ\hat\epsilon$ is
not a weak equivalence. Combining the decompositions~\eqref{eq190114a} and~\eqref{eq190114b}, we obtain the next decomposition of $\hat f$.
$$(x_n)\xra[\notin\catw(A)]{\widehat{\id_{[P]}}^{(z_p)}_{(z_p')}\circ\hat\epsilon}(z_p')\xra{\hat\phi\circ\widehat{\id_{[P]}}^{(z_p')}_{(a_pz_p)}}(y_m).$$
Since $\hat f$ is weakly irreducible and the first factor here is not a weak equivalence, it follows that the second factor must be a weak
equivalence. Proposition~\ref{prop170503c} implies that $a_p$ is invertible in $A$ for all $p\neq p_0$.
This establishes Claim~1.

Because of~\eqref{eq190114a}, the morphism $\widehat{\id_{[P]}}^{(z_p)}_{(a_pz_p)}$ is weakly irreducible by Lemma~\ref{lem190115a}.
However, it decomposes in~\eqref{eq190114b}, so another application of Lemma~\ref{lem190115a} shows that the morphism
$\widehat{\id_{[P]}}^{(z_p)}_{(z_p')}$ is weakly irreducible.

Claim 2: the element $a_{p_0}\in A$ is irreducible. By Claim~1, this element is non-invertible.
Suppose that $a_{p_0}=bc$ for some $b,c\in A$.
Then the weakly irreducible morphism $\widehat{\id_{[P]}}^{(z_p)}_{(z_p')}$ from~\eqref{eq190114b} factors as 
$$(z_p) \xra{\widehat{\id_{[P]}}} (z_1,\ldots,z_{p_0-1},bz_{p_0},\ldots,z_P)\xra{\widehat{\id_{[P]}}} 
(z_1,\ldots,z_{p_0-1},bcz_{p_0},\ldots,z_P)$$
so the fact that this composition is 
weakly
irreducible implies that
one of the morphisms in the composition
is a weak equivalence.
By definition, it follows that $b$ is invertible in $A$ or $c$ is invertible. This establishes Claim~2.

Recall that Proposition~\ref{prop170423a} implies that $r=u\prod_pa_p=a_{p_0}(u\prod_{p\neq p_0}a_p)$ for some invertible $u\in A$.
Claim~1 implies that the element $u\prod_{p\neq p_0}a_p$ is invertible in $A$.
Since $a_{p_0}$ is irreducible in $A$ by Claim~2 and $u\prod_{p\neq p_0}a_p$ is invertible, their product $r$ is irreducible in $A$, as desired.
\end{proof}

Our next result is for use in Theorem~\ref{prop170501a}.

\begin{prop}\label{prop170501b}
Let $P\in\bbn$, and
consider a divisibility morphism
$\widehat{\id_{[P]}}\colon(z_p)\to(a_pz_p)$ in $\catf(A)$. 
Such a morphism is 
weakly
irreducible if and only if there is an integer $p_0\in[P]$ such that
$a_{p_0}$ is irreducible in $A$ and for all $p\neq p_0$ the element $a_p$ is invertible in~$A$,
that is, if and only if the element $r=\prod_pa_p$ is irreducible in $A$.
\end{prop}

\begin{proof}
The element $r$ satisfies $\prod_p(a_pz_p)=r\prod_pz_p$, so Theorem~\ref{lem170501a} says that
$\widehat{\id_{[P]}}$ is weakly irreducible if and only if $r$ is irreducible in $A$, that is, if and only if 
there is an integer $p_0\in[P]$ such that
$a_{p_0}$ is irreducible in $A$ and for all $p\neq p_0$ the element $a_p$ is invertible in~$A$.
\end{proof}

Next, we give a characterization of irreducibility akin to Proposition~\ref{prop170503c}.

\begin{thm}\label{prop170501a}
Let $\hat f\colon(x_n)\to (y_m)$ be a morphism in $\catf(A)$ between non-empty tuples.
Consider the decomposition $\hat f=\hat\phi\circ\widehat{\id_{[P]}}\circ\hat\epsilon$ from Proposition~\ref{prop170423a}
$$(x_n)\xra{\hat\epsilon}(z_p)\xra{\widehat{\id_{[P]}}}(a_pz_p)
\xra{\hat\phi}(y_m).$$
Then 
the morphism $\hat f$ is 
weakly
irreducible
if and only if the divisibility morphism
$\widehat{\id_{[P]}}$ is 
weakly
irreducible (see Proposition~\ref{prop170501b}).
\end{thm}

\begin{proof}
Assume that the divisibility morphism
$\widehat{\id_{[P]}}$ is 
weakly
irreducible. Proposition~\ref{prop170501b} implies that $\prod_pa_p$ is irreducible in $A$.
By Remark~\ref{disc171015b}, let $r\in A$ be the unique element of $A$ such that $\prod_my_m=r\prod_nx_n$.
Since $r=u\prod_na_n$ for some invertible $a\in A$ by Proposition~\ref{prop170423a}, we conclude that $r$ is irreducible in $A$, so
Theorem~\ref{lem170501a} says that $\hat f$ is weakly
irreducible.

The converse is proved similarly, or by Lemma~\ref{lem190115a}.
\end{proof}

\begin{cor}\label{cor170501a}
A tuple $(x_n)\neq\emptytuple$ is 
weakly
irreducible in $\catf(A)$
if and only if
there is an integer $n_0\in[N]$ such that
$x_{n_0}$ is irreducible in $A$ and for all $n\neq n_0$ the element $x_n$ is invertible in $A$,
that is, if and only if $\prod_nx_n$ is irreducible in $A$.
\end{cor}

\begin{proof}
By definition, a tuple $(x_n)\neq\emptytuple$ is 
weakly
irreducible in $\catf(A)$
if and only if the morphism $\hat f\colon (1)\to (x_n)$ is 
weakly
irreducible.
Theorem~\ref{lem170501a} says that this holds if and only if $\prod_nx_n$ is irreducible in $A$, 
that is, if and only if there is an integer $n_0\in[N]$ such that
$x_{n_0}$ is irreducible in $A$ and for all $n\neq n_0$ the element $x_n$ is invertible in $A$.
\end{proof}

The next lemma is for use in the proof of Theorem~\ref{prop170503a}.

\begin{lem}\label{lem170503a}
Consider 
two
tuples $(x_n),(y_m)$ in $\catf(A)$,
and assume that in $\catf(A)$
there is a finite sequence
$(x_n)\to\cdots\to(y_m)$ of 
weakly
irreducible morphisms and weak equivalences.
If $\prod_nx_n$ is invertible or has an irreducible factorization in $A$,
then 
$\prod_my_m$ is invertible or has an irreducible factorization in $A$.
\end{lem}

\begin{proof}
Let $q$ denote the number of morphisms in the given sequence
$(x_n)\to\cdots\to(y_m)$. 
By 
Lemma~\ref{lem190110a} and Theorem~\ref{lem170501a}
there is a sequence
$r_1,\ldots,r_q$ of invertible elements and irreducible elements such that $\prod_my_m=r_1\cdots r_q\prod_nx_n$.
As $\prod_nx_n$ is assumed to be 
invertible or to have an irreducible factorization in $A$, it follows from this formula that 
$\prod_my_m$ is invertible or has an irreducible factorization in $A$.
\end{proof}

\section{Weak
Primeness}
\label{sec170702a}

\begin{assumptions}
In this section, again  $A$ is a divisibility monoid,
and we say that a morphism $\hat f$ is a weak equivalence provided that $\hat f\in\catw(A)$.
\end{assumptions}

Here
we introduce and study a notion of 
weak primeness for objects and morphisms in $\catf(A)$, 
based on the idea of weak divisibility from Section~\ref{sec170708a}.
Our treatment here is roughly parallel to our discussion of 
weak
irreducibility.

\begin{defn}\label{defn170501ezzz}
A monoid element $a\in A$ is \emph{prime} if it is not invertible and for all $b,c\in A$ if $a\mid bc$ in $A$, then either $a\mid b$
or $a\mid c$ in $A$.
We shall say that a morphism $\hat{f}$ in $\mathcal{F}(A)$ is 
\textit{weakly
prime} when $\hat{f}$ is not a weak equivalence and for every 
weak divisibility relation 
$\hat{f} \divs \hat{g} \circ \hat{h}$ either $\hat{f} \divs \hat{g}$ or $\hat{f} \divs \hat{h}$.
We say an object $(x_n) \in \mathcal{F}(A)$ is 
\textit{weakly
prime} if  the morphism $(1) \to (x_n)$ is 
weakly prime.
\end{defn}

\begin{disc}
\label{disc170501czzz}
If $\hat f\colon(x_n)\to(y_m)$ is a morphism in $\catf(A)$ with $(x_n)=\emptytuple$ or $(y_m)=\emptytuple$,
then $\hat f$ is not 
weakly
prime; so $\emptytuple$ is not 
weakly prime
in $\catf(A)$.

Similar to the previous section, the term ``weakly prime morphism'' is meant to suggest the usual notion of primeness,
modified by weak equivalences and weak divisibility.
One could introduce variations on this,
for instance, using isomorphisms or one of the divisibility notions discussed in 
Example~\ref{ex170706azzz}.
However, these versions do not suit our purposes.
\end{disc}

The next result shows, e.g., that the morphism $\widehat{\id_{[1]}}\colon(2)\to(6)$ from
Example~\ref{ex170706azzz}.
is weakly prime.
And it gives yet another indication of how the morphisms in $\catw(A)$ can see factorization properties in $A$.

\begin{thm}\label{lem170501azzz} 
For an element $r \in A$, the following conditions are equivalent.
\begin{enumerate}[\rm(i)]
\item \label{lem170501azzz1}
$r$ is prime in $A$,
\item \label{lem170501azzz5}
each morphism $\hat f\colon (x_n)\to (y_m)$ where $r\cdot\prod_nx_n=\prod_my_m$ is weakly prime,
\item \label{lem170501azzz2} 
the 1-tuple $(r)$ is weakly prime in $\mathcal{F}(A)$, 
and
\item \label{lem170501azzz0}
some morphism $\hat f\colon (x_n)\to (y_m)$ with $r\cdot\prod_nx_n=\prod_my_m$ is weakly prime.
\end{enumerate}
\end{thm}

\begin{proof}
The implications $\eqref{lem170501azzz5}\implies\eqref{lem170501azzz2}\implies\eqref{lem170501azzz0}$
are trivial.

$\eqref{lem170501azzz0}\implies\eqref{lem170501azzz1}$
Assume that
$\hat f\colon (x_n)\to (y_m)$ is a weakly prime morphism with $r\cdot\prod_nx_n=\prod_my_m$.
In particular, $r$ is not invertible in $A$, otherwise 
$\hat f$
would be a weak equivalence by Lemma~\ref{lem190110a}.
Let $a,b\in A$ be such that $r\mid ab$.
Consider the
following morphisms between 1-tuples.
$$(1)\xra[=\hat h]{\widehat{\id_{[1]}}}(a)\xra[=\hat g]{\widehat{\id_{[1]}}}(ab)$$
Because of Theorem~\ref{prop170708a}, the assumption $r\mid ab$
implies that $\hat f\divs\hat h\circ\hat g$.
Since $\hat f$ is 
weakly
prime, we have $\hat f\divs\hat h$ or $\hat f\divs\hat g$.
In the case $\hat f\divs\hat h$, another application of Theorem~\ref{prop170708a} implies that $r\mid a$.
In the case $\hat f\divs\hat g$, we conclude similarly that $r\mid b$.
Thus, the element $r$ is prime in $A$, as desired.

$\eqref{lem170501azzz1}\implies\eqref{lem170501azzz5}$
Assume that $r$ is prime in $A$. In particular, the element $r$ is not invertible in $A$.
Let  $\hat f\colon (x_n)\to (y_m)$ be a morphism such that $r\cdot\prod_nx_n=\prod_my_m$.
Lemma~\ref{lem190110a} implies that $\hat f$ is not a weak equivalence.
If $\hat f\divs\hat \alpha\circ\hat \beta$ for some morphisms $\hat\alpha$ and $\hat\beta$, 
then reverse the steps of the previous paragraph to conclude that $\hat f\divs\hat \alpha$ or $\hat f\divs\hat \beta$,
so that $\hat f$ is weakly prime, as desired.
\end{proof}

The next result is included for use in Theorem~\ref{prop170501azzz}.

\begin{prop}\label{prop170501bzzz}
With $I\in\bbn$, consider a divisibility morphism
$\widehat{\id_{[I]}}\colon(v_i)\to(a_iv_i)$ in $\catf(A)$. 
Such a morphism is 
weakly
prime if and only if there is an integer $i_0\in[I]$ such that
$a_{i_0}$ is prime in $A$ and for all $i\neq i_0$ the element $a_i$ is invertible in $A$, that is, if and only if
the element $\ti a=\prod_ia_i$ is prime in $A$.
\end{prop}

\begin{proof}
Argue as for Proposition~\ref{prop170501b}, using Theorem~\ref{lem170501azzz}.
\end{proof}

Next,
we give a characterization of 
weak primeness
akin to Proposition~\ref{prop170503c} and Theorem~\ref{prop170501a}.

\begin{thm}\label{prop170501azzz}
Let $\hat f\colon(x_n)\to (y_m)$ be a morphism in $\catf(A)$ between non-empty tuples.
Consider the decomposition $\hat f=\hat\phi\circ\widehat{\id_{[P]}}\circ\hat\epsilon$ from Proposition~\ref{prop170423a}
$$(x_n)\xra{\hat\epsilon}(z_p)\xra{\widehat{\id_{[P]}}}(a_pz_p)
\xra{\hat\phi}(y_m).$$
Then 
the morphism $\hat f$ is 
weakly
prime
if and only if the divisibility morphism
$\widehat{\id_{[P]}}$ is 
weakly
prime (see Proposition~\ref{prop170501bzzz}).
\end{thm}

\begin{proof}
Argue as for Theorem~\ref{prop170501a}, using Theorem~\ref{lem170501azzz} and Proposition~\ref{prop170501bzzz}.
\end{proof}

\begin{cor}\label{cor170501azzz}
A tuple $(x_n)\neq\emptytuple$ is 
weakly
prime in $\catf(A)$
if and only if
there is an integer $n_0\in[N]$ such that
$x_{n_0}$ is prime in $A$ and for all $n\neq n_0$ the element $x_n$ is invertible in $A$,
that is, if and only if $\prod_nx_n$ is prime in $A$.
\end{cor}

\begin{proof}
Argue as for Corollary~\ref{cor170501a}, using
Theorem~\ref{lem170501azzz}.
\end{proof}

\begin{cor}\label{cor170710a}
If a morphism $\hat f$ is 
weakly
prime in $\catf(A)$, then it is 
weakly
irreducible in $\catf(A)$.
If a tuple $(x_n)$ is 
weakly
prime in $\catf(A)$, then it is 
weakly
irreducible in $\catf(A)$.
\end{cor}

\begin{proof}
Since prime elements of $A$ are irreducible, 
the first implication follows from Theorems~\ref{lem170501a} and~\ref{lem170501azzz},
and the second implication is by Corollaries~\ref{cor170501a} and~\ref{cor170501azzz}.
\end{proof}

\section{Factorization in Integral Domains}\label{accp-implies-atomic}

\begin{assumptions}
In this section, Let $D$ be an integral domain.
We say that a morphism $\hat f$ in $\catf(D\setminus\{0\})$ is a weak equivalence provided that $\hat f\in\catw(D\setminus\{0\})$.
\end{assumptions}

Next,
we interpret some properties of integral domains  in terms of the category of factorization.
This entire section fits with the theme that factorization properties of our monoid are mirrored in the morphisms of its
category of factorization.
We begin
with the atomic property.
Recall that our integral domain $D$ is \emph{atomic} if every non-zero non-unit in $D$ factors as a finite product of 
irreducible elements of $D$, also known as ``atoms''.

\begin{thm}\label{prop170503a}
The following conditions are equivalent:
\begin{enumerate}[\rm(i)]
\item\label{prop170503a1}
the 
integral
domain $D$ is atomic,
\item\label{prop170503a2'}
every morphism $\hat f\colon (x_n)\to(y_m)$ in $\catf(D \setminus \{0\})$ between non-empty tuples decomposes as a finite composition
$(x_n)\to\cdots\to(y_m)$ of 
weakly
irreducible morphisms and weak equivalences,
\item\label{prop170503a2}
every tuple $(y_m)
\neq\emptytuple$ in $\catf(D \setminus \{0\})$ admits a finite chain
$(1)\to\cdots\to(y_m)$ of 
weakly
irreducible morphisms and weak equivalences, and
\item\label{prop170503a3}
every 1-tuple $(y)$ admits a finite chain
$(1)\to\cdots\to(y)$ of 
weakly
irreducible morphisms and weak equivalences, 
and
\item\label{prop170503a3'}
every 1-tuple $(y)$ admits a finite chain
$(1)\to\cdots\to(y)$ of 
weakly
irreducible morphisms and weak equivalences consisting entirely of 1-tuples.
\end{enumerate}
\end{thm}

\begin{proof}
$\eqref{prop170503a1}\implies\eqref{prop170503a2'}$
Assume that $D$ is atomic.

Claim: each divisibility morphism $\widehat{\id_{[P]}}\colon(z_p)\to(a_pz_p)$ between $P$-tuples
decomposes as a finite composition
$(z_p)\to\cdots\to(a_pz_p)$ of 
weakly
irreducible divisibility morphisms and weak equivalences, consisting entirely of $P$-tuples.
Indeed, since $D$ is atomic, each $a_p$ factors as $a_p=u_pq_{p,1}\cdots q_{p,K_p}$ 
with $K_p\geq 0$, where $u_p$ is a unit and
each $q_{p,k}$ is irreducible.
We induct on $n=\sum_{p=1}^{p_1}K_p$.
In the base case $n=0$, each $a_p$ is a unit, so $\widehat{\id_{[P]}}$ is a weak equivalence by Proposition~\ref{prop170503c}.
For the induction step, assume that $n\geq 1$, so we have $K_{p_1}\geq 1$ for some $p_1\in [P]$.
Then the morphism $\widehat{\id_{[P]}}$ factors as follows
\begin{align*}
(z_p)
&\xra[\text{w.i.}]{\widehat{\id_{[P]}}}(q_{1,1}z_1,z_2,\ldots,z_P)
\xra[(\dagger)]{\widehat{\id_{[P]}}}(a_pz_p).
\end{align*}
The  morphism labeled ``w.i.'' is
weakly
irreducible by Theorem~\ref{lem170501a}.
The morphism labeled $(\dagger)$ satisfies our induction hypothesis, so it
decomposes as a finite composition
of 
weakly
irreducible morphisms and weak equivalences, consisting entirely of $P$-tuples.
Thus, the same is true of the original morphism $\widehat{\id_{[P]}}$.

Now,
consider a morphism $\hat f\colon (x_n)\to(y_m)$ in $\catf(D \setminus \{0\})$ between non-empty tuples.
Consider the decomposition $\hat f=\hat\phi\circ\widehat{\id_{[P]}}\circ\hat\epsilon$ from Proposition~\ref{prop170423a}
$$(x_n)\xra{\hat\epsilon}(z_p)\xra{\widehat{\id_{[P]}}}(a_pz_p)
\xra{\hat\phi}(y_m).$$
The Claim implies that the divisibility morphism $\widehat{\id_{[P]}}$
decomposes as a finite composition
$(z_p)\to\cdots\to(a_pz_p)$ of 
weakly
irreducible divisibility morphisms and weak equivalences.
Since $\hat\epsilon$ and $\hat\phi$ are weak equivalences by Remark~\ref{disc170501a}, 
the desired conclusion for $\hat f$ follows from the displayed factorization.

$\eqref{prop170503a2'}\implies\eqref{prop170503a2}$
Condition~\eqref{prop170503a2} is the special case $(x_n)=(1)$ of condition~\eqref{prop170503a2'}, so this implication is trivial.

$\eqref{prop170503a2}\implies\eqref{prop170503a3}$
Condition~\eqref{prop170503a3} is the special case $N=1$ of condition~\eqref{prop170503a2}, so this implication is trivial.

$\eqref{prop170503a3}\implies\eqref{prop170503a1}$
Assume that every 1-tuple 
$(y)$
admits a finite chain
$(1)\to\cdots\to(y)$ of 
weakly
irreducible morphisms and weak equivalences.
Let $x\in D$ be a non-zero non-unit.
By assumption, there is a finite chain
$(1)\to\cdots\to(x)$ of 
weakly
irreducible morphisms and weak equivalences.
Since $1\in A$ is invertible, Lemma~\ref{lem170503a} 
implies that $x$ is  invertible or has an irreducible factorization as well.
Thus, $D$ is atomic, as desired.

$\eqref{prop170503a1}\implies\eqref{prop170503a3'}$
The natural morphism $(1)\to(y)$ is a divisibility morphism between 1-tuples, so the Claim in the proof of the preceding implication
gives the desired conclusion.

$\eqref{prop170503a3'}\implies\eqref{prop170503a3}$
This implication is trivial.
\end{proof}

Next, we have a similar characterization of 
UFD's (unique factorization domains).
Compare 
conditions~\eqref{prop170710a5}--\eqref{prop170710a4}
to~\cite[Theorem~5.1]{coykendall:idg}.
It is worth noting that 
conditions~\eqref{prop170710a5} and~\eqref{prop170710a4}
do not include commutativity requirements on the diagrams they contain.

\begin{thm}\label{prop170710a}
The following conditions are equivalent.
\begin{enumerate}[\rm(i)]
\item\label{prop170710a1}
the 
integral
domain $D$ is a UFD.
\item\label{prop170710a2}
\emph{(a)} Every
tuple $(x_n)\neq\emptytuple$ in $\catf(D \setminus \{0\})$ admits a finite chain
$(1)\to\cdots\to(x_n)$ of 
weakly
irreducible morphisms and weak equivalences, and
\emph{(b)}
every 
weakly
irreducible tuple in $\catf(D \setminus \{0\})$ is 
weakly
prime.
\item\label{prop170710a3}
\emph{(a)} Every
morphism $\hat f\colon (x_n)\to(y_m)$ in $\catf(D \setminus \{0\})$ between non-empty tuples decomposes as a finite composition
$(x_n)\to\cdots\to(y_m)$ of 
weakly
irreducible morphisms and weak equivalences, and
\emph{(b)}
every 
weakly
irreducible morphism  in $\catf(D \setminus \{0\})$ is 
weakly
prime.
\item\label{prop170710a5}
\emph{(a)} Every morphism $\hat f\colon (x_n)\to(y_m)$ in $\catf(D \setminus \{0\})$ between non-empty tuples decomposes 
as a finite composition
$(x_n)\to\cdots\to(y_m)$ of 
weakly
irreducible morphisms and weak equivalences; and
\emph{(b)} for every pair of morphisms $(v_i)\xra{\hat f}(z_p)\xla{\hat g}(w_j)$ in $\catf(D\setminus\{0\})$ with $(v_i)$ and $(w_j)$ 
weakly
irreducible,
either there is a weak equivalence $(v_i)\to(w_j)$ or 
there  are morphisms
\begin{equation}
\label{eq170719a}
\begin{split}
\xymatrix{
&(z_p)
\\
(v_i)\ar[r]\ar[ru]^{\hat f}
&(t)\ar[u]
&(w_j)\ar[l]\ar[lu]_{\hat g}
}
\end{split}
\end{equation}
where each morphism in the bottom row of the diagram is 
weakly
irreducible and $(t)$ is a 1-tuple
such that $t$ is associate to $\prod_iv_i\cdot\prod_jw_j$ in $D$.
\item\label{prop170710a4}
\emph{(a)} Every tuple $(x_n)\neq\emptytuple$ in $\catf(D \setminus \{0\})$ admits a finite chain
$(1)\to\cdots\to(x_n)$ of 
weakly
irreducible morphisms and weak equivalences; and
\emph{(b)} for every pair of morphisms $(v)\xra{\hat f}(z)\xla{\hat g}(w)$ between 1-tuples in $\catf(D\setminus\{0\})$ with $(v)$ and $(w)$ 
weakly
irreducible,
either there is a weak equivalence $(v)\to(w)$ or 
there is a positive integer $K$ such that there are morphisms
$$\xymatrix@C=6.3mm@R=2.5mm{
&&&(z)
\\ 
(v)&&&&&&(w)
\\
(t_{0,\ell_0})\ar[r]\ar[rrruu]^{\hat f}
\ar@{=}[u]
&(t_{1,\ell_1})\ar[rruu]
&(t_{2,\ell_2})\ar[r]\ar[l]\ar[ruu]
&(t_{3,\ell_3})\ar[uu]
&\cdots\ar[r]\ar[l]
&(t_{2K-1,\ell_{2K-1}})\ar[lluu]
&(t_{2K,\ell_{2K}})\ar[l]\ar[llluu]_{\hat g}\ar@{=}[u]
}$$
where each morphism in the bottom row is 
weakly
irreducible, and
for each $k\in [K]$ 
the elements $\prod_{\ell_{2k-2}}t_{2k-2,\ell_{2k-2}}\cdot\prod_{\ell_{2k}}t_{2k,\ell_{2k}}$ and
$\prod_{\ell_{2k-1}}t_{2k-1,\ell_{2k-1}}$ are associates.
\end{enumerate}
\end{thm}

\begin{proof}
$\eqref{prop170710a1}\iff\eqref{prop170710a3}$
The 
integral
domain $D$ is a UFD if and only if it is atomic and every irreducible element of $D$ is prime. 
Theorem~\ref{prop170503a} shows that $D$ being atomic is equivalent to 
condition~\eqref{prop170710a3}(a).
If every 
weakly
irreducible morphism in $\catf(D \setminus \{0\})$ is 
weakly
prime, then 
every irreducible element of $D$ is prime by 
Theorems~\ref{lem170501a} and~\ref{lem170501azzz};
and conversely.

$\eqref{prop170710a1}\iff\eqref{prop170710a2}$
Argue as in the previous paragraph using Corollaries~\ref{cor170501a} and~\ref{cor170501azzz}.

$\eqref{prop170710a1}\implies\eqref{prop170710a5}$
Assume that $D$ is a UFD.
In particular, $D$ is atomic, so Theorem~\ref{prop170503a} shows that condition~(a) holds. 
For condition~(b), 
consider a pair of morphisms $(v_i)\xra{\hat f}(z_p)\xla{\hat g}(w_j)$ in $\catf(D\setminus\{0\})$ with $(v_i)$ and $(w_j)$ 
weakly
irreducible,
and assume that there is not a weak equivalence $(v_i)\to(w_j)$.
Since $(v_i)$ and $(w_j)$ are 
weakly
irreducible, Corollary~\ref{cor170501a} implies that there are integers $i_0\in[I]$ and
$j_0\in[J]$ such that $v_{i_0}$ and $w_{j_0}$ are irreducible in $D$, for all $i\in[I]\setminus\{i_0\}$ the element $v_i$ is a unit in $D$,
and for all $j\in[J]\setminus\{j_0\}$ the element $w_j$ is a unit in $D$.
The existence of the morphism $\hat f$ implies that $v_{i_0}\mid z_{p_1}$ for some $p_1\in[P]$,
and similarly $w_{j_0}\mid z_{p_2}$ for some $p_2\in[P]$,
as the assumption that $D$ is a UFD implies that $v_{i_0}$ and $w_{j_0}$ are prime in $D$.

Claim: The elements $v_{i_0}$ and $w_{j_0}$ are not associates in $D$.
By way of contradiction, suppose that $v_{i_0}$ and $w_{j_0}$ were  associates in $D$, 
say $u\in D$ is a unit such that $w_{j_0}=uv_{i_0}$.
Consider the next morphisms  in $\catf(D\setminus\{0\})$
where $\hat\epsilon$ drops all the units $v_i$ with $i\neq i_0$ 
and $\hat\sigma$ is the morphism induced by the condition $w_{j_0}\mid\prod_jw_j$.
$$(v_i)\xra{\hat\epsilon}(v_{i_0})\xra{\widehat{\id_{[1]}}}(uv_{i_0})=(w_{j_0})\xra{\hat\sigma}(w_j)$$
Remark~\ref{disc170501a} implies that the three morphisms in this display are weak equivalences.
Thus, the composition $\hat\sigma\circ\widehat{\id_{[1]}}\circ\hat\epsilon\colon (v_i)\to (w_j)$
is also a weak equivalence, by the 2-of-3 condition 
in Proposition~\ref{wk_equiv}.
This contradicts our assumption that there is not a weak equivalence $(v_i)\to(w_j)$,
establishing the Claim.

Set $(t)=(v_{i_0}w_{j_0})$.
Let $\hat\alpha$ be the composition of the following morphisms in $\catf(D\setminus\{0\})$
where $\hat\epsilon$ drops all the units $v_i$ with $i\neq i_0$ 
$$
(v_i)\xra{\hat\epsilon}(v_{i_0})\xra{\widehat{\id_{[1]}}}(v_{i_0}w_{j_0})=(t).
$$
Since $w_{j_0}$ is irreducible, Theorem~\ref{prop170501a} implies that $\hat\alpha$ is 
weakly irreducible. 
Similarly, there is a weakly irreducible morphism 
$(w_j)\to(t)$.
To complete the diagram~\eqref{eq170719a}, it suffices to exhibit a morphism $(t)=(v_{i_0}w_{j_0})\to(z_p)$.
In the case $p_1\neq p_2$, such a morphism comes from the  conditions $t=v_{i_0}w_{j_0}\mid z_{p_1}z_{p_2}\mid\prod_pz_p$.
In the case $p_1= p_2$, we have $v_{i_0},w_{j_0}\mid z_{p_1}$.
Since
$v_{i_0}$ and $w_{j_0}$ are non-associate irreducibles in $D$,
we conclude that $v_{i_0}w_{j_0}\mid z_{p_1}\mid\prod_pz_p$ because $D$ is a UFD.
This yields a morphism $(t)=(v_{i_0}w_{j_0})\to(z_p)$, thus
completing the diagram~\eqref{eq170719a} in this case, as well as the proof of this implication.

$\eqref{prop170710a5}\implies\eqref{prop170710a4}$
Theorem~\ref{prop170503a} shows that condition~\eqref{prop170710a5}(a) implies condition~\eqref{prop170710a4}(a).
Since condition~\eqref{prop170710a4}(b) is a special case of~\eqref{prop170710a5}(b),
this implication is established.

$\eqref{prop170710a4}\implies\eqref{prop170710a1}$
Assume that~\eqref{prop170710a4} holds.
Then condition~\eqref{prop170710a4}(a) implies that $D$ is atomic by Theorem~\ref{prop170503a}.
We use~\cite[Theorem~5.1]{coykendall:idg}
to show that $D$ is a UFD.
To this end, let $z,v,w\in D$ be such that $v$ and $w$ are non-associate irreducible divisors of $z$. 
It
suffices to exhibit irreducible divisors $a_1,\ldots,a_m$ of $z$
such that the following products all divide $z$: $va_1$, $a_1a_2$, \ldots, $a_{m-1}a_m$, and $a_mw$.
(In the notation of~\cite[Theorem~5.1]{coykendall:idg},
this will show that $G(z)$ is connected.)

The assumptions $v\mid z$ and $w\mid z$ provide morphisms $(v)\xra{\widehat{\id_{[1]}}}(z)\xla{\widehat{\id_{[1]}}}(w)$
in $\catf(D\setminus\{0\})$.
As $v$ and $w$ are non-associate irreducibles, there is not a weak equivalence $(v)\to(w)$.
So, condition~\eqref{prop170710a4}(b)
provides a positive integer $K$ and~morphisms
$$\xymatrix@C=6.3mm@R=2.5mm{
&&&(z)
\\ 
(v)&&&&&&(w)
\\
(t_{0,\ell_0})\ar[r]\ar[rrruu]^{\widehat{\id_{[1]}}^{(v)}_{(z)}}
\ar@{=}[u]
&(t_{1,\ell_1})\ar[rruu]
&(t_{2,\ell_2})\ar[r]\ar[l]\ar[ruu]
&(t_{3,\ell_3})\ar[uu]
&\cdots\ar[r]\ar[l]
&(t_{2K-1,\ell_{2K-1}})\ar[lluu]
&(t_{2K,\ell_{2K}})\ar[l]\ar[llluu]_{\widehat{\id_{[1]}}^{(w)}_{(z)}}\ar@{=}[u]
}$$
where each morphism in the bottom row  is weakly irreducible and
for each $k\in [K]$ the element $\prod_{\ell_{2k-2}}t_{2k-2,\ell_{2k-2}}\cdot\prod_{\ell_{2k}}t_{2k,\ell_{2k}}$ is associate
to $\prod_{\ell_{2k-1}}t_{2k-1,\ell_{2k-1}}$.
Thus, 
Theorem~\ref{lem170501a} provides irreducible elements 
$b_1,b_2,\ldots,b_{2K}\in D$ such that
for 
all $k\in[K]$
we have
\begin{gather*}
\textstyle
b_{2k-1}\prod_{\ell_{2k-2}}t_{2k-2,\ell_{2k-2}}=\prod_{\ell_{2k-1}}t_{2k-1,\ell_{2k-1}}=b_{2k}\prod_{\ell_{2k}}t_{2k,\ell_{2k}}.
\end{gather*}

For  $i=0,\ldots,2K$, set $t_i=\prod_{\ell_i}t_{i,\ell_i}$.
Then the above equalities involving $b_i$ read 
\begin{equation}\label{eq170724a}
b_{2k-1}t_{2k-2}=t_{2k-1}=b_{2k}t_{2k}
\end{equation}
for $k\in[K]$.
Furthermore, 
the vertical/diagonal arrows in the above diagram in $\catf(D\setminus\{0\})$
show that $t_i\mid z$ for all $i\in[2K-1]$; in particular, this shows that $t_{2k-1}\mid z$ for all $k\in [K]$.

The equations~\eqref{eq170724a} in the case $k=1$ show that 
$b_1t_0=t_1=b_2t_2$.
By assumption, we have $t_1\sim t_0t_2$, where $\sim$ is the associate relation. 
Combining these relations, we find that
$b_1t_0\sim t_0t_2\sim b_2t_2$ and hence $b_1\sim t_2$ and $b_2\sim t_0$.
In particular, since $b_1$ is irreducible, so is $t_2$, that is, $t_2$ is an irreducible divisor of $z$ such that $t_2v=t_2t_0\sim t_1\mid z$,
by the previous paragraph. 
Arguing similarly, we see that for all $k\in [K]$ the element $t_{2k}$ is an irreducible divisor of $z$ such that $t_{2k-2}t_{2k}\sim t_{2k-1}\mid z$.
For $k=K$, this reads as $t_{2K-2}w\sim t_{2K-1}\mid z$.

In summary, this shows that we have irreducible divisors $t_2,\ldots,t_{2K-2}$ of $z$
such that the following products all divide $z$: $vt_2$, $t_2t_4$, \ldots, $t_{2K-4}t_{2K-2}$, and $t_{2K-2}w$.
In other words, the elements $a_k=t_{2k}$ for $k\in[K-1]$ satisfy the conditions described in the first paragraph of this part of the proof.
So $D$ is a UFD by~\cite[Theorem~5.1]{coykendall:idg}.
\end{proof}

We continue with 
similar characterizations of other properties from~\cite{AAZ},
beginning with
the ACCP (ascending chain condition on principal ideals)
property.

\begin{thm}\label{prop170503b}
The following conditions are equivalent:
\begin{enumerate}[\rm(i)]
\item\label{prop170503b1}
the 
integral
domain $D$ satisfies ACCP,
\item\label{prop170503b2}
every chain of morphisms $(x_{1,n})_{n=1}^{N_1} \leftarrow (x_{2,n})_{n=1}^{N_2} \leftarrow \cdots$ in 
$\catf(D\setminus\{0\})$ stabilizes to a chain of weak equivalences,
that is, for every such chain there is an index $i$ such that each morphism $(x_{i,n})_{n=1}^{N_i}\from(x_{i+1,n})_{n=1}^{N_{i+1}}\from\cdots$ is
a weak equivalence, and
\item\label{prop170503b3}
every chain of morphisms $(x_1) \leftarrow (x_2) \leftarrow \cdots$ of 1-tuples in 
$\catf(D\setminus\{0\})$ stabilizes to a chain of isomorphisms (equivalently, weak equivalences).
\end{enumerate}
\end{thm}

\begin{proof}
First, compare Theorem~\ref{iso-thm} and Definition~\ref{defn170501d} to see that a morphism 
$(x)\to(y)$ between 1-tuples in $\catf(D\setminus\{0\})$ is a weak equivalence if and only if it is an isomorphism. 

$\eqref{prop170503b1}\implies\eqref{prop170503b2}$
Assume that $D$ 
satisfies
ACCP, and consider a chain of morphisms 
$$(x_{1,n})_{n=1}^{N_1} \xla{\widehat{f_1}} (x_{2,n})_{n=1}^{N_2} \xla{\widehat{f_2}} \cdots$$ 
in 
$\catf(D\setminus\{0\})$.
This provides a chain of divisibility relations 
$$\textstyle\cdots \mid \prod_{n=1}^{N_2} x_{2,n} \mid  \prod_{n=1}^{N_1}x_{1,n}$$
hence a chain of principal ideals 
$\langle \prod_{n=1}^{N_1}x_{1,n}\rangle\subseteq\langle \prod_{n=1}^{N_2} x_{2,n}\rangle\subseteq\cdots$.
ACCP implies that
this chain stabilizes, so there is an integer $i$ such that
for each $j\geq i$ there is a unit $u_i\in D$ such that
$\prod_{n=1}^{N_j} x_{j,n}=u_j\prod_{n=1}^{N_{j+1}} x_{j+1,n}$.
Lemma~\ref{lem190110a}
implies that $\widehat{f_j}$ is a weak equivalence for each $j\geq i$, as desired.

$\eqref{prop170503b2}\implies\eqref{prop170503b3}$
Condition~\eqref{prop170503b3} is the special case of condition~\eqref{prop170503b2}
where $N_i=1$ for all $i$.
Thus, this implication is clear.

$\eqref{prop170503b3}\implies\eqref{prop170503b1}$
Assume that condition~\eqref{prop170503b3} is satisfied. A chain
of principal ideals $\langle x_1\rangle\subseteq\langle x_2\rangle\subseteq\cdots$
yields a chain of divisibility morphisms
$(x_1)\xla{\widehat{\id_{[1]}}} (x_2)\xla{\widehat{\id_{[1]}}}\cdots$
which stabilizes by assumption. 
As in the proof of $\eqref{prop170503b1}\implies\eqref{prop170503b2}$, 
one concludes that the original chain of principal ideals also stabilizes.
\end{proof}

Recall that  $D$ is a \emph{half-factorial domain (HFD)}
provided that it is atomic and for every pair of irreducible factorizations $p_1\cdots p_m=q_1\cdots q_n$ in $D$,
we have $m=n$. (Note that originally HFD's were not assumed to be atomic.)
Our next result characterize HFDs
as in Theorem~\ref{prop170503a};
see also Theorem~\ref{prop170918a}.

\begin{thm}\label{prop170503d}
The following conditions are equivalent:
\begin{enumerate}[\rm(i)]
\item\label{prop170503d1}
the integral domain $D$ is an HFD;
\item\label{prop170503d2'}
every morphism $\hat f\colon (x_n)\to(y_m)$ in $\catf(D \setminus \{0\})$ between non-empty tuples decomposes as a finite composition
$(x_n)\to\cdots\to(y_m)$ of 
weakly
irreducible morphisms and weak equivalences,
and for every such
morphism $\hat f$  every such decomposition
for $\hat f$ has the same number of 
weakly
irreducible morphisms;
\item\label{prop170503d2}
every tuple $(x_n)\neq\emptytuple$ in $\catf(D \setminus \{0\})$ admits a finite chain
$(1)\to\cdots\to(x_n)$ of 
weakly
irreducible morphisms and weak equivalences, and every such chain
for $(x_n)$ has the same number of 
weakly
irreducible morphisms; 
\item\label{prop170503d3}
every 1-tuple $(x)$ admits a finite chain
$(1)\to\cdots\to(x)$ of 
weakly
irreducible morphisms and weak equivalences, and every such chain
for $(x)$ has the same number of 
weakly
irreducible morphisms; and
\item\label{prop170503d3'}
every 1-tuple $(x)$ admits a finite chain
$(1)\to\cdots\to(x)$ of 
weakly
irreducible morphisms and weak equivalences consisting entirely of 1-tuples, and every such chain
for $(x)$ has the same number of 
weakly
irreducible morphisms.
\end{enumerate}
\end{thm}

\begin{proof}
We prove the implication $\eqref{prop170503d1}\implies\eqref{prop170503d2'}$;
the remainder of the proof follows as for 
Theorem~\ref{prop170503a}.
Assume that $D$ is an HFD, and let $\hat f\colon (x_n)\to(y_m)$ be a morphism in $\catf(D \setminus \{0\})$ between non-empty tuples.
Since $D$ is atomic, Theorem~\ref{prop170503a} implies that $\hat f$ 
decomposes as a finite chain
$(x_n)\to\cdots\to(y_m)$ of 
weakly
irreducible morphisms and weak equivalences.
Let $a\in A$ be such that $a\prod_nx_n=\prod_my_m$ and factor $a$ as a product of irreducibles $a=\prod_{i=1}^Ir_i$.
The decomposition of $\hat f$ in the proof of Theorem~\ref{prop170503a} has exactly $I$ 
weakly
irreducible morphisms.

Let 
$(x_n)=(x_n)_{n=1}^N\xra{\widehat{f_1}}(x_{1,n_1})_{n_1=1}^{N_1}\xra{\widehat{f_2}}\cdots\xra{\widehat{f_K}}(x_{K,n_K})_{n_K=1}^{N_K}=(y_m)$
be a decomposition of $\hat f$ as a finite chain of 
weakly
irreducible morphisms and weak equivalences.
We need to show that the number of 
weakly
irreducible morphisms in this chain is $I$.
Since each $\widehat{f_k}$ is a morphism, there is an element $s_k\in A$ such that
$\prod_{n_k}x_{k,n_k}=s_k\prod_{n_{k-1}}x_{k-1,n_{k-1}}$.
Composing the $\widehat{f_k}$'s we obtain
$$
\textstyle\prod_ks_k\prod_nx_n=
\prod_my_m=a\prod_nx_n$$
so cancellation implies that $a=\prod_ks_k$.
Since the $\widehat{f_k}$'s are   weak equivalences and 
weakly
irreducible morphisms,
Lemma~\ref{lem190110a} and Theorem~\ref{lem170501a}
imply that the 
corresponding $s_k$'s are units and irreducibles, respectively.
Since $D$ is an HFD, the number of irreducible $s_k$'s must be exactly $I$.
Since the number of 
weakly
irreducible $\widehat{f_k}$'s is the same as the number of irreducible $s_k$'s, our proof is complete.
\end{proof}

The following characterizations of 
bounded factorization domains (BFDs)
and finite factorization domains (FFDs)
are proved like 
Theorem~\ref{prop170503d}.
Recall that  $D$ is a \emph{BFD} if it is atomic, and for every non-zero non-unit $x\in D$
there is a bound (depending only on $x$) on the numbers of factors in each irreducible factorization of $x$.
Also, $D$ is an \emph{FFD} if it is atomic, and each non-zero non-unit in $D$ has only finitely many irreducible factors
(up to associates).

\begin{thm}\label{prop170710z}
The following conditions are equivalent:
\begin{enumerate}[\rm(i)]
\item\label{prop170710z1}
the integral domain $D$ is a BFD;
\item\label{prop170710z2}
every morphism $\hat f\colon (x_n)\to(y_m)$ in $\catf(D \setminus \{0\})$ between non-empty tuples decomposes as a finite composition
$(x_n)\to\cdots\to(y_m)$ of 
weakly
irreducible morphisms and weak equivalences, and for every such morphism $\hat f$ 
there is a bound (depending only on $\hat f$) on the number of 
weakly
irreducible morphisms in each such decomposition of $\hat f$;
\item\label{prop170710z3}
every tuple $(x_n)\neq\emptytuple$ in $\catf(D \setminus \{0\})$ admits a finite chain
$(1)\to\cdots\to(x_n)$ of 
weakly
irreducible morphisms and weak equivalences, and for every such tuple $(x_n)$ 
there is a bound (depending only on $(x_n)$) on the number of 
weakly
irreducible morphisms in each such chain
for $(x_n)$; 
\item\label{prop170710z4}
every 1-tuple $(x)$ admits a finite chain
$(1)\to\cdots\to(x)$ of 
weakly
irreducible morphisms and weak equivalences, and 
for every such 1-tuple $(x)$ 
there is a bound (depending only on $(x)$) on the number of 
weakly
irreducible morphisms in each such chain
for $(x)$; and
\item\label{prop170710z4'}
every 1-tuple $(x)$ admits a finite chain
$(1)\to\cdots\to(x)$ of 
weakly
irreducible morphisms and weak equivalences consisting entirely of 1-tuples, and 
for every such 1-tuple $(x)$ 
there is a bound (depending only on $(x)$) on the number of 
weakly
irreducible morphisms in each such chain
for $(x)$.
\end{enumerate}
\end{thm}

\begin{thm}\label{prop170710y}
The following conditions are equivalent:
\begin{enumerate}[\rm(i)]
\item\label{prop170710y1}
the integral domain $D$ is an FFD;
\item\label{prop170710y2}
every morphism $\hat f\colon (x_n)\to(y_m)$ in $\catf(D \setminus \{0\})$ between non-empty tuples decomposes as a finite chain
$(x_n)\to\cdots\to(y_m)$ of 
weakly
irreducible morphisms and weak equivalences,
and every such 
morphism
$\hat f$ has a finite number of non-weakly-associate weak divisors;
\item\label{prop170710y3}
every tuple $(x_n)\neq\emptytuple$ in $\catf(D \setminus \{0\})$ admits a finite chain
$(1)\to\cdots\to(x_n)$ of 
weakly
irreducible morphisms and weak equivalences, and 
the morphism
$(1)\to(x_n)$ has a finite number of non-weakly-associate weak divisors;
\item\label{prop170710y4}
every 1-tuple $(x)$ admits a finite chain
$(1)\to\cdots\to(x)$ of 
weakly
irreducible morphisms and weak equivalences, and 
the morphism
$(1)\to(x)$ has a finite number of non-weakly-associate weak divisors;
and
\item\label{prop170710y4'}
every 1-tuple $(x)$ admits a finite chain
$(1)\to\cdots\to(x)$ of 
weakly
irreducible morphisms and weak equivalences consisting entirely of 1-tuples, and 
the morphism
$(1)\to(x)$ has a finite number of non-weakly-associate weak divisors.
\end{enumerate}
\end{thm}

We conclude by souping up a characterization of HFDs due to Zaks~\cite{zaks}.
For this result, set $\bbn_0=\bbn\cup\{0\}$.

\begin{thm}\label{prop170918a}
The following conditions are equivalent.
\begin{enumerate}[\rm(i)]
\item\label{prop170918a1}
The integral domain $D$ is an HFD and not a field.
\item\label{prop170918a2}
There is a function $\zaksmor\colon\mord\to\bbn_0$
such that 
\begin{enumerate}[\rm(\protect{\ref{prop170918a2}}-1)]
\item\label{prop170918a2a}
$\im(\zaksmor)=\bbn_0$, 
\item\label{prop170918a2d}
$\zaksmor(\hat f)=0$ if and only if $\hat f$ is a weak equivalence in $\catf(D\setminus\{0\})$,
\item\label{prop170918a2b}
$\zaksmor(\hat f)=1$ if and only if $\hat f$ is weakly irreducible in $\catf(D\setminus\{0\})$,
and
\item\label{prop170918a2c}
$\zaksmor(\hat g\circ\hat f)=\zaksmor(\hat g)+\zaksmor(\hat f)$ for all $\hat f$ and $\hat g$ such that $\hat g\circ\hat f$ is defined.
\end{enumerate}
\item\label{prop170918a3}
There is a function $\zaksobj\colon\objd\to\bbn_0$
such that 
\begin{enumerate}[\rm(\protect{\ref{prop170918a3}}-1)]
\item\label{prop170918a3a}
$\im(\zaksobj)=\bbn_0$, 
\item\label{prop170918a3b}
$\zaksobj((x_n))=1$ if and only if $(x_n)$ is  
weakly
irreducible in $\catf(D\setminus\{0\})$,
\item\label{prop170918a3c}
$\zaksobj((x_n)\otimes(y_m))=\zaksobj((x_n))+\zaksobj((y_m))$ for all  $(x_n),(y_m)$, and
\item\label{prop170918a3d}
for each morphism $\hat f\colon(x_n)\to(y_m)$ we have $\zaksobj((x_n))\leq\zaksobj((y_m))$,
with equality holding if and only if $\hat f$ is a weak equivalence in $\catf(D\setminus\{0\})$.
\end{enumerate}
\end{enumerate}
\end{thm}

\begin{proof}
\eqref{prop170918a1}$\implies$\eqref{prop170918a2}
Assume that $D$ is an HFD and not a field. 
In particular, $D$ is atomic.
For each $x\in D\setminus\{0\}$, 
let $\zakselt(x)=n$ where $x=up_1\cdots p_n$ with $u$ a unit of $D$ and each $p_i$ an irreducible.
The fact that $D$ is an HFD implies that $\zakselt$ is well defined (i.e., independent of the choice of factorization).
By the proof of~\cite[Lemma~1.3]{zaks}, this function has the following properties:
\begin{enumerate}[(\protect{\ref{prop170918a1}}-1)]
\item $\im(\zakselt)=\bbn_0$,
\item $\zakselt(x)=0$ if and only if $x$ is a unit of $D$, 
\item $\zakselt(x)=1$ if and only if $x$ is irreducible, and
\item $\zakselt(xy)=\zakselt(x)+\zakselt(y)$ for all non-zero $x,y\in D$.
\end{enumerate}
Condition~(\ref{prop170918a1}-1) is where we use the assumption that $D$ is not a field.

For each morphism $\hat f\colon(x_n)\to(y_m)$ 
in
$\catf(D\setminus\{0\})$,
recall that there is a unique element $a\in D\setminus\{0\}$ such that $\prod_my_m=a\prod_nx_n$;
define $\zaksmor(\hat f)=\zakselt(a)\in\bbn_0$.
If $\hat g\colon(y_m)\to(z_p)$ is another 
morphism,
the element $b\in D\setminus\{0\}$
satisfying $\prod_pz_p=b\prod_my_m$ therefore also satisfies $\prod_pz_p=ab\prod_nx_n$,
so we have
\begin{align*}
\zaksmor(\hat g\circ\hat f)
&=\zakselt(ab)
=\zakselt(a)+\zakselt(b)
=\zaksmor(\hat f)+\zaksmor(\hat g).
\end{align*}
Thus, condition~(\ref{prop170918a2}-4) is satisfied.

For condition~(\ref{prop170918a2}-1),
since we have
$\im(\zakselt)=\bbn_0$, let $n\in \bbn_0$ be given and let $x\in D\setminus\{0\}$ be such that
$\zakselt(x)=n$. By definition, it follows that the morphism $\widehat{\id_{[1]}}\colon(1)\to(x)$ satisfies
$\zaksmor(\widehat{\id_{[1]}})=\zakselt(x)=n$. Thus, we have $\im(\zaksmor)=\bbn_0$.

For condition~(\ref{prop170918a2}-3), consider a morphism $\hat f\colon(x_n)\to(y_m)$, with 
$a\in D\setminus\{0\}$ such that $\prod_my_m=a\prod_nx_n$.
By 
Theorem~\ref{lem170501a},
the morphism $\hat f$ is weakly irreducible if and only if $a$ is irreducible in $D$,
that is, if and only if $\zaksmor(\hat f)=\zakselt(a)=1$.

Condition~(\ref{prop170918a2}-2) is verified similarly,  
with Lemma~\ref{lem190110a}.
This completes the proof of the implication \eqref{prop170918a1}$\implies$\eqref{prop170918a2}.

\eqref{prop170918a1}$\implies$\eqref{prop170918a3}
Assume that $D$ is an HFD,
and not a field. 
We use the function $\zakselt$ and the properties~(\ref{prop170918a1}-1)--(\ref{prop170918a1}-4)
from the proof of the 
implication \eqref{prop170918a1}$\implies$\eqref{prop170918a3}.

For 
each
tuple $(x_n)$ of $\catf(D\setminus\{0\})$, set
$\zaksobj((x_n))=\zakselt(\prod_nx_n)$.
Thus, the definition $(x_n)\otimes(y_m)=(x_1,\ldots,x_N,y_1,\ldots,y_M)$ explains the first equality in the next display,
and the second equality is from the properties of $\zakselt$.
\begin{align*}
\textstyle
\zaksobj((x_n)\otimes(y_m))
&\textstyle=\zakselt(\prod_nx_n\cdot\prod_my_m)\\
&\textstyle=\zakselt(\prod_nx_n)+\zakselt(\prod_my_m)\\
&\textstyle=\zaksobj((x_n))+\zaksobj((y_m))
\end{align*}
The third equality is by definition. 
This explains condition~(\ref{prop170918a3}-3).

For condition~(\ref{prop170918a3}-1), let $n\in\bbn_0$
 be given and let $x\in D\setminus\{0\}$ be such that
$\zakselt(x)=n$. By definition, it follows that the 1-tuple $(x)$ satisfies
$\zaksobj((x))=\zakselt(x)=n$. Thus, we have $\im(\zaksmor)=\bbn_0$.

For condition~(\ref{prop170918a3}-2), consider a tuple 
$(x_n)$.
Corollary~\ref{cor170501a} implies that
$(x_n)$ is 
weakly
irreducible in $\catf(D\setminus\{0\})$
if and only if
$\prod_nx_n$ is irreducible in $D$, that is, if and only if $\zaksobj((x_n))=\zakselt(\prod_nx_n)=1$.

For condition~(\ref{prop170918a3}-4), consider a morphism $\hat f\colon(x_n)\to(y_m)$,
and let $a\in D\setminus\{0\}$ be such that $\prod_my_m=a\prod_nx_n$.
Thus, we have
\begin{align*}
\zaksobj((y_m))
&\textstyle
=\zakselt(\prod_my_m)
\textstyle
=\zakselt(a\prod_nx_n)\\
&\textstyle
=\zakselt(a)+\zakselt(\prod_nx_n)
\textstyle
\geq\zakselt(\prod_nx_n)
\textstyle
=\zaksobj((x_n)).
\end{align*}
Moreover, equality holds in the penultimate step here if and only if $\zakselt(a)=0$,
that is, if and only if $a$ is a unit,
that is, if and only if $\hat f$ is a weak equivalence by
Lemma~\ref{lem190110a}.
This completes the proof of the implication \eqref{prop170918a1}$\implies$\eqref{prop170918a3}.

\eqref{prop170918a2}$\implies$\eqref{prop170918a1}
Let $\zaksmor$ be given as in condition~\eqref{prop170918a2}.
We argue as in the proof of~\cite[Lemma~1.3]{zaks} to show that $D$ is an HFD,
and not a field.

Claim: $D$ is atomic.
Let $\hat f\colon (x_n)\to(y_m)$ be a morphism in $\mord$.
According to Theorem~\ref{prop170503a}, we need to show that $\hat f$ factors as a  composition
of finitely many weakly irreducible morphisms and weak equivalences.
We verify this condition by strong induction on $L=\zaksmor(\hat f)$.
If $L=0$ or 1, then we are done by condition~(\ref{prop170918a2}-2) or~(\ref{prop170918a2}-3), respectively.
This addresses the base case. 
For the induction step, assume that $L=\zaksmor(\hat f)\geq 2$ and that each morphism $\hat g$ in
$\mord$ such that $\zaksmor(\hat g)<L$ factors as a  composition
of finitely many weakly irreducible morphisms and weak equivalences.
Conditions~(\ref{prop170918a2}-2) and~(\ref{prop170918a2}-3) imply that $\hat f$ is not a weak equivalence and
is not weakly irreducible. Thus, there are morphisms $\hat g$ and $\hat h$ in $\mord$ that are not weak equivalences
and such that $\hat f=\hat h\circ\hat g$. 
Our assumptions about $\zaksmor$ imply that $\zaksmor(\hat g),\zaksmor(\hat h)\geq 1$ and 
$\zaksmor(\hat h)+\zaksmor(\hat g)=\zaksmor(\hat h\circ\hat g)=\zaksmor(\hat f)\geq 2$.
In particular, we have $\zaksmor(\hat g),\zaksmor(\hat h)<\zaksmor(\hat f)=L$,
so by our induction assumption, the morphisms $\hat g$ and $\hat h$ both factor as   compositions
of finitely many weakly irreducible morphisms and weak equivalences, hence,
so does their composition $\hat f$.
This confirms the claim.

Now, let $\hat f\colon (x_n)\to(y_m)$ be a morphism in $\mord$,
and use what we have just established to find a finite list $\widehat{p_1},\ldots,\widehat{p_L}$
of weakly irreducible morphisms and weak equivalences such that $\hat f=\widehat{p_1}\circ\cdots\circ\widehat{p_L}$.
Using our assumptions about $\zaksmor$, we have
$\zaksmor(\hat f)=\zaksmor(\widehat{p_1})+\cdots+\zaksmor(\widehat{p_L})$ and this equals the number of 
weakly irreducible morphisms in the list $\widehat{p_1},\ldots,\widehat{p_L}$. 
In particular, the number of weakly irreducible morphisms in this list is independent of the choice of factorization of $\hat f$.
Thus, Theorem~\ref{prop170503d} implies that $D$ is an HFD.

To show that $D$ is not a field, let $\hat f\colon (x_n)\to(y_m)$ be a morphism in 
$\catf(D\setminus\{0\})$ such that $\zaksmor(\hat f)=1$, since $\im(\zaksmor)=\bbn_0$.
Condition~(\ref{prop170918a2}-3) implies that $\hat f$ is weakly irreducible; in particular, it is not a weak equivalence.
Lemma~\ref{lem190110a}
implies
that the non-zero element $a\in D$ such that $\prod_my_m=a\prod_nx_n$ is not a unit in $D$,
so $D$ is not a field.
This concludes the proof of this implication.

\eqref{prop170918a3}$\implies$\eqref{prop170918a1}
Let $\zaksobj$ be given as in item~\eqref{prop170918a3}.
For all $x\in D\setminus\{0\}$, set $\zakselt(x)=\zaksobj((x))$.
Since for all $x,y\in D\setminus\{0\}$, the factorization morphism $(xy)\to(x,y)=(x)\otimes(y)$ is a weak equivalence,
we have
\begin{align*}
\zakselt(xy)
&=\zaksobj((xy))
=\zaksobj((x)\otimes(y))
=\zaksobj((x))+\zaksobj((y))
=\zakselt(x)+\zakselt(y).
\end{align*}
Thus, condition~(\ref{prop170918a1}-4) from the first part of this proof is satisfied.
Furthermore, this implies that
$\zakselt(1)=\zakselt(1\cdot 1)=\zakselt(1)+\zakselt(1)$
and thus $\zakselt(1)=0$.

Next, we show that an element $x\in D\setminus\{0\}$ is a unit if and only if $\zaksobj((x))=0$. 
Indeed, consider the morphism $\widehat{\id_{[1]}}\colon(1)\to(x)$. 
By assumption, $\zaksobj((1))=0\leq \zaksobj((x))$.
Moreover,  $x$ is a unit in $D$ if and only if $\widehat{\id_{[1]}}$ is a weak equivalence by 
Lemma~\ref{lem190110a},
if and only if $\zaksobj((x))=\zaksobj((1))=0$ by condition~(\ref{prop170918a3}-4).
Thus, condition~(\ref{prop170918a1}-2) from the first part of this proof is satisfied.

We  verify condition~(\ref{prop170918a1}-3) similarly:
Because of Corollary~\ref{cor170501a}, an element $x\in D\setminus\{0\}$ is
irreducible in $D$ if and only if the 1-tuple $(x)$ is 
weakly
irreducible in $\catf(D\setminus\{0\})$,
that is, if and only if $\zakselt(x)=\zaksobj((x))=1$.

Next, we  verify condition~(\ref{prop170918a1}-1).
By assumption, there 
exists
a tuple $(x_n)$ in $\catf(D\setminus\{0\})$ such that $\zaksobj((x_n))=1$
by what we have already shown.
Condition~(\ref{prop170918a3}-2) implies that $(x_n)$ is 
weakly
irreducible in $\catf(D\setminus\{0\})$,
so the element $x=\prod_nx_n$ is irreducible in $D$ by Corollary~\ref{cor170501a}.
Hence, $\zakselt(x)=1$. Condition~(\ref{prop170918a1}-4) implies that $\zakselt(x^k)=k$ for all $k\in\bbn$.
As $\zakselt(1)=0$, it follows that $\im(\zakselt)=\bbn_0$, as~desired.

Define a function $\ell$ from the set of non-zero non-units of $D$ to $\bbn$ by the formula $\ell(x)=\zakselt(x)$,
i.e., $\ell$ is obtained from $\zakselt$ by restricting the domain and codomain.
Conditions~(\ref{prop170918a3}-1)--(\ref{prop170918a3}-4) show that $\ell$ is well-defined and satisfy the following:
\begin{enumerate}[(1)]
\item $\im(\ell)=\bbn$,
\item $\ell(x)=1$ if and only if $x$ is irreducible, and
\item $\ell(xy)=\ell(x)+\ell(y)$ for all non-zero $x,y\in D$.
\end{enumerate}
These are exactly the conditions that allow us to apply~\cite[Lemma~1.3]{zaks} to conclude that $D$ is an HFD.
(Note that conditions~(1) and~(2) imply that $D$ has an irreducible element, in particular, a non-unit, so $D$ is not a field.)
\end{proof}

\section*{Acknowledgments}
We are grateful to Jim Coykendall 
and George Janelidze
for their thoughtful suggestions about this work.

\providecommand{\bysame}{\leavevmode\hbox to3em{\hrulefill}\thinspace}
\providecommand{\MR}{\relax\ifhmode\unskip\space\fi MR }
\providecommand{\MRhref}[2]{%
  \href{http://www.ams.org/mathscinet-getitem?mr=#1}{#2}
}
\providecommand{\href}[2]{#2}

\bibliographystyle{amsplain}

\end{document}